\documentclass{article}

\usepackage{geometry}
\usepackage[utf8]{inputenc}
\usepackage[english]{babel}
\usepackage{amsmath,amsfonts,amsthm,amssymb,mathtools}
\usepackage{tikz}
\usetikzlibrary{patterns}
\usetikzlibrary{hobby}
\usepackage{microtype}
\usepackage{exscale}
\usepackage{ifthen}
\usepackage{graphicx}
\usepackage{cite}
\usepackage{caption}
\usepackage{float}
\usepackage[hidelinks]{hyperref}
\usepackage{cleveref}
\usepackage{bbm}
\usepackage{bm}
\usepackage{varwidth}
\usepackage{caption}

\usepackage{algorithmic}
\usepackage[vlined,boxruled,linesnumbered,algo2e,algosection]{algorithm2e}
\newcommand{\SetAlgorithmStyle}{
  \setcounter{AlgoLine}{0}
  \SetKwData{Left}{left}\SetKwData{This}{this}\SetKwData{Up}{up}
  \SetKwInOut{Input}{Input}
  \SetKwInOut{Output}{Output}
  \ResetInOut{input}
  \SetKwComment{tcp}{//}{}
  \SetKwFor{For}{for}{}{end}
  \SetKwFor{Fun}{function}{}{end}
  \SetArgSty{}
  \DontPrintSemicolon
}

\DeclareMathAlphabet{\mathpzc}{OT1}{pzc}{m}{it}

\newtheorem{prop}{Proposition}[section]
\newtheorem{defi}[prop]{Definition}
\newtheorem{lemma}[prop]{Lemma}
\newtheorem{theo}[prop]{Theorem}

\newtheorem{remark}[prop]{Remark}

\newtheorem{example}[prop]{Example}

\newcommand{\set}[2]{\left\{#1\,\middle|\,#2\right\}}

\newcommand{\bset}[2]{\big\{#1\,\big|\,#2\big\}}
\newcommand{\bigset}[2]{\bigl\{#1\,\big|\,#2\bigr\}}

\newcommand{\N}{\mathbb{N}}

\newcommand{\R}{\mathbb{R}}
\newcommand{\C}{\mathbb{C}}

\newcommand{\eps}{\varepsilon}
\newcommand{\e}{\mathrm{e}}
\newcommand{\iu}{\mathrm{i}}

\newcommand{\A}{\mathcal{A}}

\renewcommand{\H}{\mathcal{H}}

\newcommand{\So}{S_{\H_1}}
\newcommand{\St}{S_{\H_2}}
\newcommand{\dd}{\mathrm{d}}
\newcommand{\dt}{\frac{\dd}{\dd t}}
\newcommand{\E}{\mathbb{E}}
\newcommand{\trace}{\mathrm{trace}}

\newcommand{\dist}{\mathrm{dist}}
\newcommand{\dH}{\mathrm{d}_\mathrm{H}}
\newcommand{\norm}[1]{\lVert #1 \rVert}

\newcommand{\U}{\mathcal{U}}
\renewcommand{\a}{a_{\varphi}}
\renewcommand{\b}{b_{\psi,\varphi}}
\renewcommand{\c}{c_{\varphi,\psi}}
\renewcommand{\d}{d_{\psi}}
\newcommand{\lam}{\lambda_{\varphi,\psi}}
\renewcommand{\phi}{\varphi}

\newcommand{\m}{\mathrm{m}}

\allowdisplaybreaks

\title{Computing the Quadratic Numerical Range}

\author{Birgit Jacob, Lukas Vorberg, Christian Wyss}
                     
\begin{document}

\maketitle 

\begin{abstract}
A novel algorithm for the computation of the quadratic numerical range is presented and exemplified yielding much better results in less time compared to the random vector sampling method. Furthermore, a bound on the probability for the random vector sampling method to produce a point exceeding a neighborhood of the expectation value in dependence on norm and size of the matrix is given.
\end{abstract}

\section{Introduction}

Knowledge about the eigenvalues of a matrix is a powerful tool in analysis and numerics and has a wide range of applications such as stability analysis of linear dynamical systems. Unfortunately, they are very sensible to model uncertainties and their exact determination can come at a very high computational cost which is why supersets that are a lot easier to compute while still preserving important information have become an attractive topic of ongoing research.

A well-established and thoroughly studied superset is the numerical range of a matrix $\A\in\C^{n\times n}$ defined by
\begin{equation*}
W(\A) = \left\{\langle \A x,x\rangle \mid \norm{x} = 1\right\},
\end{equation*}
see \cite{GustafsonRao, HornJohnson, Kato95}. Here, $\langle\cdot,\cdot\rangle$ denotes the scalar product of $\C^n$. The numerical range is always convex which has advantages when it comes to determining whether or not all eigenvalues are contained in a half-plane or when an algorithm for its computation is to be found. See \cite{CowenHarel, Johnson, LoiselMaxwell} for effective algorithms. But the convexity has its drawbacks as well because spectral gaps can not be detected.

In \cite{LangerTretter} Langer and Tretter introduced the quadratic numerical range (QNR) as a new concept to enclose the spectrum of a block operator matrix in a Hilbert space. In the matrix case --- depending on a decomposition $\C^{n} = \C^{n_{1}}\oplus\C^{n_{2}}\eqqcolon \H_{1}\oplus\H_{2}$ such that
\begin{equation*}
  \A = \begin{bmatrix}A & B\\C & D\end{bmatrix}
\end{equation*}
with $A,B,C,D$ acting in and between $\H_{1}$ and $\H_{2}$ --- the QNR is defined as the set of eigenvalues of the $2\times 2$ matrices
\begin{equation*}
\begin{bmatrix} \langle Ax,x\rangle & \langle By,x\rangle \\ \langle Cx,y\rangle & \langle Dy,y\rangle \end{bmatrix}
  \end{equation*}
  with $x\in \H_1$, $y\in\H_{2}$, $\norm{x}=\norm{y}=1$. This provides a closed subset of the numerical range that still encloses all the eigenvalues of $\A$, is not necessarily convex and consists of at most two connected components which need not be convex either, see \cite{LangerMarkusMatsaevTretter} and the monograph \cite{Tretter}, where many more properties are proven as well. For applications of the QNR, we refer to \cite{FrommerJacobKahlWyssZwaan}, \cite{JacobTretterTrunkVogt} and \cite{Linden} where the superset is exploited for Krylov type methods, damped systems and the location of zeros of polynomials. In \cite{muh-mar12} and \cite{muh-mar13} approximation schemes for possibly unbounded operators are established and convergence theorems are proven relating the QNR of an operator to the QNR of its finite-dimensional discretizations.

  The development of effective algorithms for the numerical computation of the QNR faces several challenges, e.g. because of the lack of convexity and ideas approved for the numerical range can not be adapted straightforwardly. So far the only method for computing the quadratic numerical range is based on random vector sampling, see \cite{Fazlollahi} for a Matlab implementation. This method however comes with various disadvantages. We show that especially for higher dimensional matrices the computed points will very likely cluster in a small subset of each component making convergence very slow and expensive. We present a new approach for the computation of the quadratic numerical range which is more deterministic than random vector sampling and based on the idea of seeking the boundary by maximization of an objective function.
  Multiple examples illustrate that in this way much better results are obtained in less time.

  This article is organized as follows: In \Cref{sect:QNR} we recall the definition of the quadratic numerical range and examine curves therein. \Cref{sect:algorithm} contains the algorithm for the computation of the QNR alongside explanations for the chosen procedure. This algorithm is then exemplified in \Cref{sect:examples} where it is compared to the random vector sampling method. In \Cref{sect:concentration} we show that the probability of a point in the QNR generated by the random vector method to be outside of a small neighborhood of the expected value decays exponentially with an increase of the dimension of the matrix when its norm stays constant.

  We conclude this introduction with some remarks on the notation used.
  The real and imaginary parts of a complex number $z$ are denoted by $\Re z$ and $\Im z$ respectively and $\C_+\coloneqq \set{z\in\C}{\Re z >0}$ and $\C_-\coloneqq \set{z\in\C}{\Re z <0}$ are the right and left half planes in $\C$.
  $\norm{\cdot}$ denotes either the $2$-norm of a vector or the operator norm of a matrix induced by the $2$-norm. For a square matrix $\A$, $\sigma(\A)$ is the set of its eigenvalues and $\varrho(\A) \coloneqq \C\setminus\sigma(\A)$. Furthermore, $\partial K$ denotes the boundary of a set $K\subset\C$ and for a some $\lambda\in\C$ and $\eps>0$ we define $B_\eps(\lambda)\coloneqq \set{z\in\C}{|\lambda-z|<\eps}$.

\section{The Quadratic Numerical Range and curves therein}\label{sect:QNR}

Throughout this article we consider matrices $\A\in\C^{n\times n}$ with a block decomposition of the form
\begin{equation*}
  \A = \begin{bmatrix}
    A & B\\
    C & D
  \end{bmatrix}
\end{equation*}
where $A\colon\H_1\to \H_1$, $B\colon\H_2\to\H_1$, $C\colon\H_1\to\H_2$, $D\colon\H_2\to\H_2$ and $\H_1\oplus \H_2 \coloneqq \C^{n_{1}}\oplus\C^{n_{2}} = \C^n$. Note that every matrix can be written in such a form once a decomposition $\C^n = \H_{1}\oplus\H_{2}$ is chosen.

\begin{defi}
  The \emph{quadratic numerical range (QNR)} is given by
  \begin{equation*}
    W^2(\A) = \bigcup_{x\in S_{\H_1}, y\in S_{\H_2}} \sigma\left(\begin{bmatrix} \langle Ax,x\rangle & \langle By,x\rangle \\ \langle Cx,y\rangle & \langle Dy,y\rangle \end{bmatrix}\right),
  \end{equation*}
  where $S_{\H_i} = \{x\in \H_i \mid \norm{x}=1\}$, $i=1,2$.
\end{defi}

In other words, the QNR consists of the solutions $\lambda$ of the quadratic equations
\begin{equation}\label{eq:quadratic}
\lambda^2-(\langle Ax,x\rangle+\langle Dy,y\rangle)\lambda + \langle Ax,x\rangle\langle Dy,y\rangle-\langle By,x\rangle\langle Cx,y\rangle = 0
\end{equation}
with $(x,y)\in\So\times\St$.

\begin{remark}
$W^{2}(\A)$ contains all eigenvalues of $\A$ and consists of at most two (connected) components, see \cite{LangerMarkusMatsaevTretter}.
\end{remark}

In order to shorten the notation we use the abbreviations
\begin{equation}\label{eq:abcd}
     M_{x,y} \coloneqq \begin{bmatrix} a_x & b_{y,x} \\ c_{x,y} & d_y \end{bmatrix} \coloneqq \begin{bmatrix} \langle Ax,x\rangle & \langle By,x\rangle \\ \langle Cx,y\rangle & \langle Dy,y\rangle \end{bmatrix}\in\C^{2\times 2}
   \end{equation}
   for $(x,y)\in\So\times\St$.

   \begin{prop}\label{thm:complex_root}
     Let $(x_0,y_0)\in\So\times\St$ be such that $\sigma(M_{x_0,y_0})$ consists of two simple eigenvalues. Then there exists a neighborhood $U\subset\So\times\St$ of $(x_0,y_0)$ such that $\bigcup_{(x,y)\in U}\sigma(M_{x,y})$ consists of two disconnected components $W_1$ and $W_2$ that can be separated by a straight line and there exists $\theta_0\in [0,2\pi]$ and a branch of the complex root $\sqrt{\cdot}\colon G\to\C$ with $G = \C\setminus\{r\e^{\iu\theta_0}  \mid  r\geq 0\}$ such that
  \begin{equation}\label{eq:W1}
    W_1 = \set{\frac{a_x+d_y}{2} + \sqrt{\left(\frac{a_x-d_y}{2}\right)^2+b_{y,x}c_{x,y}}}{(x,y)\in U}
  \end{equation}
  and
  \begin{equation}\label{eq:W2}
    W_2 = \set{\frac{a_x+d_y}{2} - \sqrt{\left(\frac{a_x-d_y}{2}\right)^2+b_{y,x}c_{x,y}}}{(x,y)\in U}.
  \end{equation}
\end{prop}

\begin{proof}
From \cite[Theorem 2.5.1]{Kato95} we have that the eigenvalues of a matrix depend continuously on its entries. Therefore, there exists a neighborhood $U\subset\So\times\St$ of $(x_0,y_0)$ such that $\bigcup_{(x,y)\in U}\sigma(M_{x,y})$ consists of two disconnected components $W_1$ and $W_2$ that can be separated by a straight line. Without loss of generality, we assume that $W_1$ and $W_2$ are separated by the imaginary axis because considering the shifted and rotated matrix $\e^{\iu\theta}(\A + zI)$ for some $\theta\in[0,2\pi]$ and $z\in\C$ would lead to the computation of the eigenvalues of
  \begin{equation*}
    \begin{bmatrix} \tilde{a}_x & \tilde{b}_{y,x} \\ \tilde{c}_{x,y} & \tilde{d}_y \end{bmatrix} = \e^{\iu\theta}\begin{bmatrix} a_x+z & b_{y,x} \\ c_{x,y} & d_y+z \end{bmatrix},
  \end{equation*}
  so the radicant would be given by $\left(\frac{\tilde{a}_x-\tilde{d}_y}{2}\right)^2+\tilde{b}_{y,x}\tilde{c}_{x,y} = \e^{2\iu\theta}\left(\left(\frac{a_x-d_y}{2}\right)^2+b_{y,x}c_{x,y}\right)$ and thus $\tilde{G} = \e^{2\iu\theta}G$.\\
  So let us assume $W_1\subset\C_+$ and $W_2\subset\C_-$ and let $\lambda_1\in W_1$ and $\lambda_2\in W_2 $ be eigenvalues of $M_{x,y}$ for given $(x,y)\in U$, i.e.~$\Re\lambda_2<0<\Re\lambda_1$ and $\lambda_1$ and $\lambda_2$ are solutions of
  \begin{equation*}
    (a_x-\lambda)(d_y-\lambda)-b_{y,x}c_{x,y} = 0 \iff \left(\lambda-\frac{a_x+d_y}{2}\right)^2 = \left(\frac{a_x-d_y}{2}\right)^2+b_{y,x}c_{x,y}.
  \end{equation*}
  Then there is a solution $q\in\C$ of $q^2 = \left(\frac{a_x-d_y}{2}\right)^2+b_{y,x}c_{x,y}$ such that $\lambda_1 = \frac{a_x+d_y}{2} + q$ and $\lambda_2 = \frac{a_x+d_y}{2}-q$.  
  It follows
  \begin{equation*}
    0<\Re(\lambda_1-\lambda_2) = 2\Re q
  \end{equation*}
  and we conclude that $q\in\C_+$ and thus $q^2 \in \C\setminus \R_{\leq 0}$. So by defining $G\coloneqq\C\setminus\R_{\leq0} = \C\setminus \{r\e^{\iu\pi}  \mid  r\geq 0\}$ and $\sqrt{\cdot}\colon G \to \C$ as the principal branch of the complex root with $\Re\sqrt{z}>0$ for all $z\in G$ we obtain
  \begin{equation*}
    \lambda_1 = \frac{a_x+d_y}{2} + \sqrt{\left(\frac{a_x-d_y}{2}\right)^2+b_{y,x}c_{x,y}} \quad \text{ and } \quad \lambda_2 = \frac{a_x+d_y}{2} - \sqrt{\left(\frac{a_x-d_y}{2}\right)^2+b_{y,x}c_{x,y}}.
  \end{equation*}
\end{proof}

\begin{remark}
Note, that the assumption in \Cref{thm:complex_root} on $\sigma(M_{x_0,y_0})$ to consist of two simple eigenvalues is fulfilled for every $(x_0,y_0)\in\So\times\St$ if $W^2(\A)$ consists of two disconnected components. Furthermore, we have $U=\So\times\St$ if the two components of $W^2(\A)$ can be separated by a straight line. In this case, the formulas in \eqref{eq:W1} and \eqref{eq:W2} can be used to match each of the two eigenvalues of a matrix $M_{x,y}$ to a specific component.
\end{remark}

In the following we will consider curves in the QNR, i.e.~continuous mappings $\lambda_{\varphi,\psi}$ from an interval $I$ into $W^2(\A)$ which are generated from continuous curves $\varphi\colon I\to\So$ and $\psi\colon I\to\St$ such that $\lambda_{\varphi,\psi}(t)$ solves \eqref{eq:quadratic} with $\phi(t)$ in place of $x$ and $\psi(t)$ in place of $y$ for all $t\in I$. We are interested in the derivative of such a curve in the QNR and in order to shorten the notation in the upcoming formulas we will henceforth and similarly to \eqref{eq:abcd} use the abbreviations
   \begin{equation*}
     M_{\phi,\psi}\coloneqq\begin{bmatrix} \a & \b \\ \c & \d \end{bmatrix} \coloneqq \begin{bmatrix} \langle A\varphi,\varphi\rangle & \langle B\psi,\varphi\rangle \\ \langle C\varphi,\psi\rangle & \langle D\psi,\psi\rangle \end{bmatrix}\colon I\to \C^{2\times 2}
   \end{equation*}
   for curves $(\phi,\psi)\colon I\to\So\times\St$.

\begin{theo}\label{thm:lambda_dot}
Let $(x_0,y_0)\in\So\times\St$ be such that $\sigma(M_{x_0,y_0})$ consists of two simple eigenvalues. Let $t_1>0$ and $(\varphi,\psi)\colon [0,t_1]\to \So\times\St$, $t\mapsto (\varphi(t),\psi(t))$, with $(\phi(0),\psi(0)) = (x_0,y_0)$ be a differentiable curve in $\So\times\St$. Then there exists a $0<t_0\leq t_1$ such that $\sigma(M_{\phi,\psi})\colon [0,t_0]\to\C^2$ consists of two differentiable curves. Denote by $\lam\colon [0,t_0]\to \C$ one of these two curves. Then
  \begin{equation*}
    \dt{\lam} = \left\langle S(\varphi,\psi,\lam)\begin{bmatrix}\varphi\\\psi\end{bmatrix},\begin{bmatrix}\dot{\varphi}\\ \dot{\psi} \end{bmatrix}\right\rangle + \left\langle S(\varphi,\psi,\lam)\begin{bmatrix}\dot{\varphi}\\ \dot{\psi}\end{bmatrix},\begin{bmatrix}\varphi\\\psi\end{bmatrix}\right\rangle
\end{equation*}
with
\begin{equation*}
  S(\varphi,\psi,\lam) = \frac{1}{2\lam-\a-\d}\begin{bmatrix} (\lam-\d)A & \c B \\ \b C & (\lam-\a)D\end{bmatrix}.
\end{equation*}
\end{theo}

\begin{proof}
  From \Cref{thm:complex_root} we have that there exists a neighborhood $U\subset\So\times\St$ of $(x_0,y_0)$ such that $\bigcup_{(x,y)\in U}\sigma(M_{x,y})$ consists of two disconnected components $W_1$ and $W_2$ that can be described in a differentiable dependence on the $(x,y)\in U$ via the formulas in \eqref{eq:W1} and \eqref{eq:W2}. We therefore obtain that $\sigma(M_{\phi,\psi})\colon [0,t_0]\to\C^2$ consists of two differentiable curves by choosing $t_0>0$ such that $(\phi(t),\psi(t))\in U$ for all $t\in[0,t_0]$.\\
  An eigenvalue $\lam(t)$ of $M_{\phi,\psi}(t)$, $t\in [0,t_0]$, satisfies
  \begin{equation*}
    (\lam(t)-\a(t))(\lam(t)-\d(t))-\b(t)\c(t) = 0,
  \end{equation*}
  so that upon differentiation we get
  \begin{equation}\label{eq:lambda_diff}
      \left(\dt{\lam}-\dt{\a}\right)(\lam-\d) + (\lam-\a)\left(\dt{\lam}-\dt{\d}\right) - \dt{\b}\c - \b\dt{\c} = 0.
    \end{equation}
  Herein
  \begin{align*}
    &\dt{\a} = \langle A\dot{\varphi},\varphi\rangle + \langle A\varphi,\dot{\varphi}\rangle\\
    &\dt{\b} = \langle B\dot{\psi},\varphi\rangle + \langle B\psi,\dot{\varphi}\rangle\\
    &\dt{\c} = \langle C\dot{\varphi},\psi\rangle + \langle C\varphi,\dot{\psi}\rangle\\
    &\dt{\d} = \langle D\dot{\psi},\psi\rangle + \langle D\psi,\dot{\psi}\rangle,
  \end{align*}
  which transforms \eqref{eq:lambda_diff} into
  \begin{align}\label{eq:diff}
    \begin{split}
    (2\lam-\a-\d)\dt{\lam} =& \left\langle \begin{bmatrix} (\lam-\d)A & \c B \\ \b C & (\lam-\a)D\end{bmatrix}\begin{bmatrix}\varphi\\\psi\end{bmatrix},\begin{bmatrix}\dot{\varphi}\\ \dot{\psi} \end{bmatrix}\right\rangle \\
    &+ \left\langle \begin{bmatrix} (\lam-\d)A & \c B \\ \b C & (\lam-\a)D\end{bmatrix}\begin{bmatrix}\dot{\varphi}\\ \dot{\psi}\end{bmatrix},\begin{bmatrix}\varphi\\\psi\end{bmatrix}\right\rangle.
    \end{split}
  \end{align}
  Now the fact that $\bigcup_{t\in[0,t_0]}\sigma(M_{\phi,\psi}(t))$ consists of two disconnected components
  implies that $\lambda_{\phi,\psi}(t) \neq \frac{\a+\d}{2}(t)$ for all $t\in[0,t_0]$ because the sum of the eigenvalues of $M_{\phi,\psi}(t)$ is equal to the sum of its diagonal entries. This allows us to divide \eqref{eq:diff} by $2\lambda_{\phi,\psi}-\a-\d$, yielding the desired formula for the derivative of $\lambda_{\phi,\psi}$.
\end{proof}

\section{An Algorithm for Computing the QNR}\label{sect:algorithm}

Our goal is to develop an algorithm for the computation of the quadratic numerical range that does not only rely on random vector sampling. This means, that we want to make a choice on the utilized vectors $(x,y)\in\So\times\St$ such that the image resulting from the point cloud of eigenvalues of the matrices $M_{x,y}$ is a very good approximation of the image of the actual QNR even for a small number of vectors. We will therefore place particular emphasis on those $(x,y)\in\So\times\St$ that correspond to boundary points of $W^2(\A)$.

\subsection{Seeking the Boundary}\label{sect:seekboundary}

Starting at a given point $\lambda_{0}\in W^{2}(\A)$ with corresponding $(x_0,y_0)\in\So\times\St$ we wish to gradually compute a sequence $(x_n,y_n)_{n\in\N}\subset\So\times\St$ such that the associated $(\lambda_{n})_{n\in\N}$ in the quadratic numerical range converge towards the boundary. In order to do so we first have to declare a direction in which we want to approach the boundary, so in the following, we will therefore start by focusing on moving parallel to the positive real axis since every other direction can be easily reduced to this case by a rotation of the matrix $\A$.

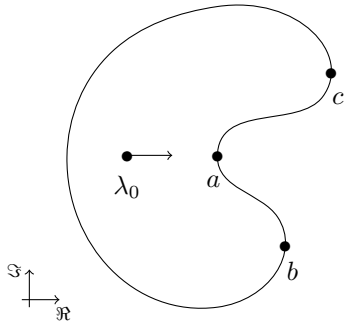
\begin{figure}[b]
  \begin{centering}
\begin{tikzpicture}
\draw (0,-2) to [closed, curve through = {(-2,0) (-1.4,1.4) (0,2) (1.4,.8)  (0,0)
  (.8,-.8) }] (0,-2);
\node[draw=none] at (0,0) {\textbullet};
\node[anchor=center, label=below:{$a$}] at (-0.05,0) {};
\node[draw=none] at (-1.2,0) {\textbullet};
\node[anchor=center, label=below:{$\lambda_0$}] at (-1.2,0) {};
\node[draw=none] at (1.51,1.1) {\textbullet};
\node[anchor=center, label=below:{$c$}] at (1.6,1.1) {};
\node[draw=none] at (.9,-1.2) {\textbullet};
\node[anchor=center, label=below:{$b$}] at (1,-1.15) {};
\draw[->] (-1.2,0.02) -- (-0.6,0.02);
\draw[->] (-2.5,-2) -- (-2.5,-1.5);
\draw[->] (-2.6,-1.9) -- (-2.1,-1.9);
\node[anchor=center, label=below:{\scriptsize$\Re$}] at (-2.05,-1.75) {};
\node[anchor=center, label=above:{\scriptsize$\Im$}] at (-2.7,-1.85) {};
\end{tikzpicture}
\captionof{figure}{The boundary of the QNR might have concave parts}
\label{fig:concave}
\end{centering}
\end{figure}

If we would leave it at aiming for $\Re\lambda_{n+1}\geq \Re\lambda_n$ for every $n\in\N$ however, we could face the problem of missing out on points on concave parts of the boundary, cf. \Cref{fig:concave}, where starting from $\lambda_0$ an algorithm that only focuses on maximization of the real part would eventually either move towards $b$ or towards $c$ but has no reason to stop at $a$.
We overcome this problem by seeking a sequence $(\lambda_n)_{n\in\N}$ that satisfies
\begin{equation*}
  \Re\lambda_{n+1} - p\big(\Im(\lambda_{n+1}-\lambda_0)\big)^2 \geq \Re\lambda_n - p\big(\Im(\lambda_n-\lambda_0)\big)^2
\end{equation*}
with a given penalty constant $p>0$ for all $n\in\N$. More precisely, we will consider the objective function $f_{\alpha,\lambda_0}\colon\So\times\St\to\R$, given by
\begin{equation}\label{eq:objective_function}
  f_{\alpha,\lambda_0}(x,y) = \Re\lambda^{(\alpha)}_{x,y}-p\big(\Im(\lambda^{(\alpha)}_{x,y}-\lambda_0)\big)^2
\end{equation}
for some $\alpha\in\left[0,2\pi\right[$, and aim to construct a sequence $(x_n,y_n)_{n\in\N}\subset\So\times\St$ such that $f_{\alpha,\lambda_0}(x_n,y_n)$ increases with $n$.
Here and from now on, $\lambda^{(\alpha)}_{x,y}$ specifically denotes the one of the two eigenvalues $\lambda^{(\alpha)}_{x,y}$ and $\tilde{\lambda}^{(\alpha)}_{x,y}$ of $M_{x,y}$ such that
\begin{align*}
  \begin{split}
    \Re(\e^{\iu\alpha}\lambda^{(\alpha)}_{x,y}) &> \Re(\e^{\iu\alpha}\tilde{\lambda}^{(\alpha)}_{x,y}), \qquad \text{if} \quad \Re(\e^{\iu\alpha}\lambda^{(\alpha)}_{x,y}) \neq \Re(\e^{\iu\alpha}\tilde{\lambda}^{(\alpha)}_{x,y}),\\
    \text{or} \qquad \Im(\e^{\iu\alpha}\lambda^{(\alpha)}_{x,y}) &\geq \Im(\e^{\iu\alpha}\tilde{\lambda}^{(\alpha)}_{x,y}), \qquad \text{if} \quad  \Re(\e^{\iu\alpha}\lambda^{(\alpha)}_{x,y})=\Re(\e^{\iu\alpha}\tilde{\lambda}^{(\alpha)}_{x,y}).
  \end{split}
\end{align*}
Note, that if $W^2(\A)$ consists of two disconnected components that can be separated by a straight line, $\alpha$ can be chosen such that each component is either the set of all $\lambda^{(\alpha)}_{x,y}$ or the set of all $\lambda^{(\alpha+\pi)}_{x,y}$ with $(x,y)\in\So\times\St$.

\Cref{fig:levelsets} illustrates the effect of the penalty constant $p$ on the level sets of $f_{\alpha,\lambda_0}$ for different choices of $p$, showing that a larger $p$ leads to a significant narrowing of the search direction. The picture also indicates, that in practice $p$ has to be chosen in dependence on the size and shape of the QNR.

\begin{figure}[h]
  \begin{centering}
\begin{tikzpicture}
\draw[->] (2.5,-2) -- (2.5,-1.5);
\draw[->] (2.4,-1.9) -- (2.9,-1.9);
\node[anchor=center, label=below:{\scriptsize$\Re$}] at (2.95,-1.75) {};
\node[anchor=center, label=above:{\scriptsize$\Im$}] at (2.3,-1.85) {};
\draw (5,-2) to [closed, curve through = {(3,0) (3.6,1.4) (5,2) (6.4,.8)  (5,0)
  (5.8,-.8) }] (5,-2);
\node[draw=none] at (3.8,0) {\textbullet};
\node[anchor=center, label=below:{$\lambda_0$}] at (3.8,0) {};
\draw[scale=1, domain=0:2.01, smooth, variable=\y]  plot ({.4*\y*\y+3.8}, {\y});
\draw[scale=1, domain=-1.93:0, smooth, variable=\y]  plot ({.4*\y*\y+3.8}, {\y});
\draw[scale=1, domain=0:1.93, smooth, variable=\y]  plot ({.4*\y*\y+4.3}, {\y});
\draw[scale=1, domain=-1.77:0, smooth, variable=\y]  plot ({.4*\y*\y+4.3}, {\y});
\draw[scale=1, domain=0:1.8, smooth, variable=\y]  plot ({.4*\y*\y+4.8}, {\y});
\draw[scale=1, domain=-1.55:0, smooth, variable=\y]  plot ({.4*\y*\y+4.8}, {\y});
\draw[scale=1, domain=0.47:1.6, smooth, variable=\y]  plot ({.4*\y*\y+5.3}, {\y});
\draw[scale=1, domain=-0.48:-1.23, smooth, variable=\y]  plot ({.4*\y*\y+5.3}, {\y});
\draw[scale=1, domain=0.57:1.31, smooth, variable=\y]  plot ({.4*\y*\y+5.8}, {\y});
\draw (10,-2) to [closed, curve through = {(8,0) (8.6,1.4) (10,2) (11.4,.8)  (10,0)
  (10.8,-.8) }] (10,-2);
\node[draw=none] at (8.8,0) {\textbullet};
\node[anchor=center, label=below:{$\lambda_0$}] at (8.8,0) {};
\draw[scale=1, domain=0:.333, smooth, variable=\y]  plot ({12*\y*\y+8.8}, {\y});
\draw[scale=1, domain=-.337:0, smooth, variable=\y]  plot ({12*\y*\y+8.8}, {\y});
\draw[scale=1, domain=0:.252, smooth, variable=\y]  plot ({12*\y*\y+9.3}, {\y});
\draw[scale=1, domain=-.257:0, smooth, variable=\y]  plot ({12*\y*\y+9.3}, {\y});
\draw[scale=1, domain=0:.132, smooth, variable=\y]  plot ({12*\y*\y+9.8}, {\y});
\draw[scale=1, domain=-.137:0, smooth, variable=\y]  plot ({12*\y*\y+9.8}, {\y});
\end{tikzpicture}
\captionof{figure}{Level sets of $f_{\alpha,\lambda_0}$ for $p$ small (left) and $p$ large (right)}
\label{fig:levelsets}
\end{centering}
\end{figure}
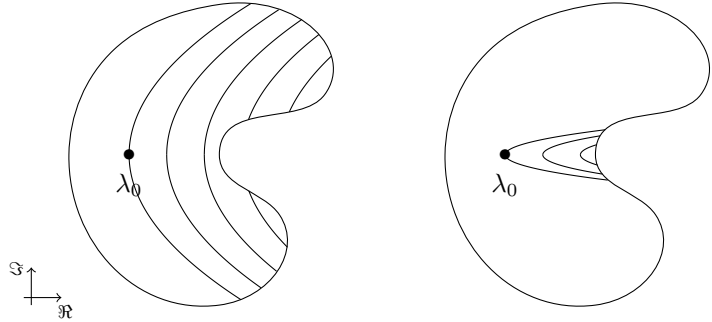

This dependence will be specified in the next result, which can be interpreted as follows: If a boundary point $\lambda_\partial$ with the same imaginary part as $\lambda_0$ satisfies the additional condition, that there exists an $r>0$ such that for every $0<\eps<r$ there exists a $\delta>0$ such that
\begin{equation}\label{eq:boundary}
	\set{z\in\C}{\frac{\eps}{2}<\Re(z-\lambda_\partial)<\eps, |\Im(z-\lambda_\partial)|<\delta}\cap W^2(\A) = \varnothing
\end{equation}
holds, our strategy of creating a sequence in $\So\times\St$ for which the value of the objective function increases can in fact result in the obtainment of $\lambda_\partial$ up to a small error if $p$ is chosen large enough.
\Cref{fig:condition} illustrates condition \eqref{eq:boundary}.

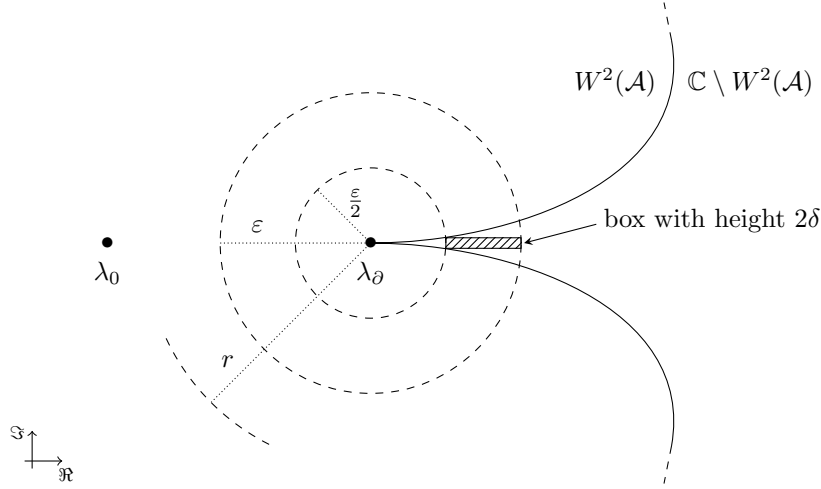
\begin{figure}
  \begin{centering}
    \begin{tikzpicture}
\draw[->] (-4.5,-2) -- (-4.5,-1.5);
\draw[->] (-4.6,-1.9) -- (-4.1,-1.9);
\node[anchor=center, label=below:{\scriptsize$\Re$}] at (-4.05,-1.75) {};
\node[anchor=center, label=above:{\scriptsize$\Im$}] at (-4.7,-1.85) {};
\draw (4,-1.7) to [out=80,in=0]  (0,1)
to [out=0, in=-80] (4,3.7);
\draw [dashed] (4,-1.7) -- (3.9,-2.2);
\draw [dashed] (4,3.7) -- (3.9,4.2);
\draw [dashed] (0,1) circle (2);
\draw [dashed] (0,1) circle (1);
\node[anchor=center, label=above:{$W^2(\A)$}] at (3.25,2.7) {};
\node[anchor=center, label=above:{$\C\setminus W^2(\A)$}] at (5.05,2.7) {};
\node[anchor=center, label=below:{$\lambda_\partial$}] at (0,1) {};
\node[draw=none] at (0,1) {\textbullet};
\draw [densely dotted] (-2,1) -- (0,1);
\node[anchor=center, label=above:{$\frac{\eps}{2}$}] at (-.2,1.1) {};
\draw [densely dotted] (-.7,1.7) -- (0,1);
\node[anchor=center, label=above:{$\eps$}] at (-1.5,.9) {};
\draw [pattern=north east lines] (1,.93) rectangle (2,1.07);
\draw [-stealth] (3,1.25) -- (2.04,1);
\node[anchor=center, label=right:{box with height $2\delta$}] at (2.85,1.3) {};
\node[anchor=center, label=below:{$\lambda_0$}] at (-3.5,1) {};
\node[draw=none] at (-3.5,1) {\textbullet};
\draw [dashed,domain=205:245] plot ({3*cos(\x)}, {3*sin(\x)+1});
\draw [densely dotted] (-2.12,-1.12) -- (0,1);
\node[anchor=center, label=above:{$r$}] at (-1.9,-.9) {};
\end{tikzpicture}
\captionof{figure}{$\lambda_\partial$ satisfies condition \eqref{eq:boundary}}
\label{fig:condition}
\end{centering}
\end{figure}
  
\begin{prop}\label{thm:local_max}
  Let $\lambda_0\in W^{2}(\A)$ and $\lambda_\partial\in\partial W^{2}(\A)$ with $\Im\lambda_\partial = \Im\lambda_0$ and suppose that $\lambda_\partial$ satisfies condition \eqref{eq:boundary} for some $r>0$.
  Then for every $\eps>0$ there exist $p>0$ and $(x_{\m},y_{\m})\in \So\times\St$ such that for all but up to one $\alpha\in\left[0,\pi\right[$ at least one of the objective functions $f_{\alpha,\lambda_0}$ or $f_{\alpha+\pi,\lambda_0}$ given by \eqref{eq:objective_function} has a local maximum in $(x_{\m},y_{\m})$ and $\lambda^{(\alpha)}_{x_{\m},y_{\m}}\in B_\eps(\lambda_\partial)\cap\partial W^2(\A)$ or $\lambda^{(\alpha+{\pi})}_{x_{\m},y_{\m}}\in B_\eps(\lambda_\partial)\cap\partial W^2(\A)$ respectively.
\end{prop}

\begin{proof}
  Let $\eps>0$. Without loss of generality, we assume that $\lambda_\partial = 0$ and therefore also $\Im\lambda_0 = 0$ by applying a shift to $\A$.
  Moreover, we will assume that $\eps<r$, where $r>0$ is the constant for which \eqref{eq:boundary} holds. Hence, there exists a $\delta>0$ such that
  \begin{equation}\label{eq:boundary0}
    \set{z\in\C}{\frac{\eps}{2}<\Re z<\eps, |\Im z|<\delta}\cap W^2(\A) = \varnothing
  \end{equation}
  and for which we assume that $\delta<\frac{\sqrt{3}}{2}\eps$.

  Let us define the set
\begin{equation*}
  K\coloneqq \set{z\in\C}{\left(\Re z>\frac{\eps}{2} \land |\Im z|<\delta\right)\lor\Re z>\eps}
\end{equation*}
  and consider the function
\begin{equation*}
    F\colon\C\to\R,\quad F(z) \coloneqq \Re z-p(\Im z)^2.
  \end{equation*}
  By choosing $p>\frac{\eps}{\delta^2}>\frac{4}{3\eps}$, we have $(\Im z)^2<\delta^2< \frac{3}{4}\eps^2$ if $F(z)\geq0$ and $\Re z\leq\eps$ and obtain
\begin{equation*}\label{eq:f=0}
  \bigset{z\in\C\setminus K}{F(z) \geq 0}\subset B_\eps(0).
\end{equation*}
Hence, the restriction of the continuous function $F$ to $\C\setminus K$ has a local maximum in $B_\eps(0)$.

This can be used in the context of the quadratic numerical range because due to \eqref{eq:boundary0} we have
\begin{equation*}
  W^{2}(\A)\cap B_\eps(0)\subset\C\setminus K
\end{equation*}
and we also know that $\set{z\in\C\setminus K}{F(z)\geq 0}\cap W^{2}(\A)$ is non-empty because of $0=\lambda_\partial\in W^{2}(\A)$ and $F(0)=0$. Therefore, the restriction of $F$ to the closed set $W^{2}(\A)$ has a local maximum at some $\lambda\in W^{2}(\A)\cap B_\eps(0)$ and there exist $(x_{\m},y_{\m})\in\So\times\St$ such that $\lambda$ is an eigenvalue of $M_{x_{\m},y_{\m}}$.

Furthermore, $\lambda\in\partial W^2(\A)$ holds because if we assume otherwise, there exists a $\gamma_{\mathrm{max}}>0$ such that $\lambda + \gamma\in W^2(\A)$ and $F(\lambda+\gamma) = F(\lambda)+\gamma>F(\lambda)$ for every $\gamma\in\left]0,\gamma_{\mathrm{max}}\right[$ which is a contradiction to $\lambda$ being a local maximum.

From \cite[Theorem 2.5.1]{Kato95} we know that the eigenvalues of a matrix depend continuously on its entries, so if $\lambda$ is the only eigenvalue of $M_{x_{\m},y_{\m}}$, there exists a neighborhood $U$ of $(x_{\m},y_{\m})$ such that $\lambda^{(\alpha)}_{x,y}\in B_{\eps}(0)$ and $\lambda^{(\alpha+\pi)}_{x,y}\in B_{\eps}(0)$ for all $(x,y)\in U$ and all $\alpha\in[0,\pi[$. Hence, both $f_{\alpha,\lambda_0}$ and $f_{\alpha+\pi,\lambda_0}$ have a local maximum in $(x_{\m},y_{\m})$ and $\lambda^{(\alpha)}_{x_{\m},y_{\m}} = \lambda^{(\alpha+\pi)}_{x_{\m},y_{\m}} = \lambda\in B_{\eps}(0)\cap\partial W^2(\A)$ for all $\alpha\in\left[0,\pi\right[$.

In the other case, if $\lambda$ is one of two distinct eigenvalues of $M_{x_{\m},y_{\m}}$, we choose $\alpha\in[0,\pi[$ such that
\begin{equation*}
\Re(\e^{\iu\alpha}\lambda^{(\alpha)}_{x_{\m},y_{\m}})\neq\Re(\e^{\iu\alpha}\lambda^{(\alpha+\pi)}_{x_{\m},y_{\m}})
\end{equation*}
and again by continuity we will find a neighborhood $U$ of $(x_{\m},y_{\m})$ such that $\Re\lambda_{1}\neq\Re\lambda_{2}$ for all $\lambda_{1}\in\bset{\e^{\iu\alpha}\lambda^{(\alpha)}_{x,y}}{(x,y)\in U}$ and $\lambda_{2}\in\bset{\e^{\iu\alpha}\lambda^{(\alpha+\pi)}_{x,y}}{(x,y)\in U}$ and either $\lambda^{(\alpha)}_{x,y}\in B_{\eps}(0)$ for all $(x,y)\in U$ if $\lambda = \lambda^{(\alpha)}_{x_{\m},y_{\m}}$ or $\lambda^{(\alpha+\pi)}_{x,y}\in B_{\eps}(0)$ for all $(x,y)\in U$ if $\lambda = \lambda^{(\alpha+\pi)}_{x_{\m},y_{\m}}$. Hence, either $f_{\alpha,\lambda_0}$ or $f_{\alpha+\pi,\lambda_0}$ has a local maximum in $(x_{\m},y_{\m})$ and $\lambda^{(\alpha)}_{x_{\m},y_{\m}} = \lambda \in B_{\eps}(0)\cap\partial W^2(\A)$ or $\lambda^{(\alpha+\pi)}_{x_{\m},y_{\m}} = \lambda \in B_{\eps}(0)\cap\partial W^2(\A)$ respectively.
\end{proof}

Let us now fix an $\alpha\in \left[0,2\pi\right[$. As explained above, we are looking for a sequence $(x_n,y_n)_{n\in\N}\subset\So\times\St$ such that $f_{\alpha,\lambda_0}(x_n,y_n)$ increases with $n$. Let us say we arrived at $(x_n,y_n)$ so far, so our goal is to find $(x_{n+1},y_{n+1})$ such that $f_{\alpha,\lambda_0}(x_{n+1},y_{n+1})\geq f_{\alpha,\lambda_0}(x_n,y_n)$.

As a first step, we will therefore identify the steepest ascent gradient of $f_{\alpha,\lambda_0}$ in $(x_n,y_n)$. Considering a differentiable curve $(\phi,\psi)\colon[0,t_1]\to\So\times\St$, $t\mapsto(\phi(t),\psi(t))$, with $(\phi(0),\psi(0)) =(x_n,y_n)$ and assuming that $\sigma(M_{x_n,y_n})$ consists of two simple eigenvalues $\lambda_n$ and $\tilde{\lambda}_n$ with $\Re(\e^{\iu\alpha}\lambda_n)>\Re(\e^{\iu\alpha}\tilde{\lambda}_n)$, we know by \Cref{thm:lambda_dot} that there exists a $0<t_0\leq t_1$ and a differentiable curve $\lambda_{\phi,\psi}\colon[0,t_0]\to W^2(\A)$ such that
\begin{equation*}
f_{\alpha,\lambda_0}\big(\phi(t),\psi(t)\big) = \Re\big(\lambda_{\phi,\psi}(t)\big)-p\big(\Im\big(\lambda_{\phi,\psi}(t)-\lambda_0\big)\big)^2
\end{equation*}
for all $t\in[0,t_0]$.
If we take a look at $\dt f_{\alpha,\lambda_0}\big(\phi(\cdot),\psi(\cdot)\big)$ at the point $t=0$, we see that again by \Cref{thm:lambda_dot}
\begin{align*}
  \dt f_{\alpha,\lambda_0}\big(\phi(0),\psi(0)\big) &= \Re\left(\dt\lambda_{\varphi,\psi}(0)\right) - 2p\Im(\lambda_n-\lambda_0)\Im\left(\dt\lambda_{\varphi,\psi}(0)\right)\\
  &= \Re \left(\left\langle S(x_n,y_n,\lambda_n)\begin{bmatrix}x_n\\y_n\end{bmatrix},\begin{bmatrix}\dot{\varphi}(0)\\ \dot{\psi}(0) \end{bmatrix}\right\rangle + \left\langle S(x_n,y_n,\lambda_n)\begin{bmatrix}\dot{\varphi}(0)\\ \dot{\psi}(0)\end{bmatrix},\begin{bmatrix}x_n\\y_n\end{bmatrix}\right\rangle\right)\\
              & \quad \ - 2p\Im(\lambda_n-\lambda_0)\Im\left(\left\langle S(x_n,y_n,\lambda_n)\begin{bmatrix}x_n\\y_n\end{bmatrix},\begin{bmatrix}\dot{\varphi}(0)\\ \dot{\psi}(0) \end{bmatrix}\right\rangle \right.\\
              &    \qquad \qquad \qquad \qquad \qquad \quad  + \left.\left\langle S(x_n,y_n,\lambda_n)\begin{bmatrix}\dot{\varphi}(0)\\ \dot{\psi}(0)\end{bmatrix},\begin{bmatrix}x_n\\y_n\end{bmatrix}\right\rangle\right)\\
                                           &= \Re\left\langle T_+(x_n,y_n,\lambda_n)\begin{bmatrix}x_n\\y_n\end{bmatrix},\begin{bmatrix}\dot{\varphi}(0)\\ \dot{\psi}(0) \end{bmatrix}\right\rangle\\
  & \quad \ - 2p\Im(\lambda_n-\lambda_0)\Im\left\langle T_-(x_n,y_n,\lambda_n)\begin{bmatrix}x_n\\y_n\end{bmatrix},\begin{bmatrix}\dot{\varphi}(0)\\ \dot{\psi}(0) \end{bmatrix}\right\rangle\\
  &= \Re\left\langle T_+(x_n,y_n,\lambda_n) + 2p\Im(\lambda_n-\lambda_0)\iu T_-(x_n,y_n,\lambda_n)\begin{bmatrix}x_n\\y_n\end{bmatrix},\begin{bmatrix}\dot{\varphi}(0)\\ \dot{\psi}(0) \end{bmatrix}\right\rangle\\
  &= \Re\left\langle T(x_n,y_n,\lambda_n,\lambda_0)\begin{bmatrix}x_n\\y_n\end{bmatrix},\begin{bmatrix}\dot{\varphi}(0)\\ \dot{\psi}(0) \end{bmatrix}\right\rangle
\end{align*}
where
\begin{align*}
  T_+(x_n,y_n,\lambda_n) &\coloneqq S(x_n,y_n,\lambda_n) + S^*(x_n,y_n,\lambda_n),\\
  T_-(x_n,y_n,\lambda_n) &\coloneqq S(x_n,y_n,\lambda_n) - S^*(x_n,y_n,\lambda_n)
\end{align*}
and
\begin{align*}\label{eq:T}
  T(x_n,y_n,\lambda_n,\lambda_0) &\coloneqq T_+(x_n,y_n,\lambda_n) + 2p\Im(\lambda_n-\lambda_0)\iu T_-(x_n,y_n,\lambda_n).
\end{align*}
Let us denote the tangential space of $\So$ in $x_n$ with regard to the real part of the inner product of $\H_1$ by $T_{x_n}\So \coloneqq \set{u\in\H_{1}}{\Re\langle x_n,u\rangle =0}$ and analogously $T_{y_n}\St \coloneqq \set{v\in\H_{2}}{\Re\langle y_n,v\rangle =0}$. Then, for $(u,v)\in T_{x_n}\So\times T_{y_n}\St$ with $\norm{u}=\norm{v}=1$, we will consider the curves $\varphi\colon [0,2\pi]\to\So$ and $\psi\colon [0,2\pi]\to\St$ defined by
\begin{equation}\label{eq:phi_psi}
\varphi(t) = \cos(t)x_n+\sin(t)u \qquad \text{and} \qquad \psi(t) = \cos(t)y_n+\sin(t)v,
\end{equation}
which satisfy $\norm{\varphi(t)}^2 = \cos^2(t)+2\Re\langle\cos(t)x_n,\sin(t)u\rangle +\sin^2(t) = 1$, $\varphi(0)=x_n$ and $\dot{\varphi}(0) = u$ as well as $\norm{\psi(t)}^2 = 1$, $\psi(0)=y_n$ and $\dot{\psi}(0) = v$.

Therefore, our problem can be simplified to finding a vector $(u,v)\in T_{x_n}\So\times T_{y_n}\St$ with $\norm{u}=\norm{v}=1$ for which the term
\begin{equation*}
\Re\left\langle T(x_n,y_n,\lambda_n,\lambda_0)\begin{bmatrix}x_n\\y_n\end{bmatrix},\begin{bmatrix}u\\ v \end{bmatrix}\right\rangle
\end{equation*}
is maximized. This vector is given by the normalized orthogonal projection of
\begin{equation*}
\begin{bmatrix}w\\z\end{bmatrix}\coloneqq T(x_n,y_n,\lambda_n,\lambda_0)\begin{bmatrix}x_n\\y_n\end{bmatrix}
\end{equation*}
onto $T_{x_n}\So\times T_{y_n}\St$, i.e.
\begin{equation*}
\begin{bmatrix}\tilde{u}\\\tilde{v}\end{bmatrix} = \begin{bmatrix}w-\Re\langle w,x_n\rangle x_n\\ z-\Re\langle z,y_n\rangle y_n\end{bmatrix}, \qquad \begin{bmatrix}u\\v\end{bmatrix} = \begin{bmatrix}\tilde{u} / \norm{\tilde{u}}\\ \tilde{v} / \norm{\tilde{v}}\end{bmatrix}.
\end{equation*}

With $u$ and $v$ at hand, we will now, in a second step, search for $(s,t)\in[0,2\pi]\times[0,2\pi]$ such that $f_{\alpha,\lambda_0}\big(\phi(s),\psi(t)\big)$ is maximal with $\varphi$ and $\psi$ as in \eqref{eq:phi_psi} and we have effectively reduced our problem to a two-dimensional optimization. This yields a new pair of vectors $(x_{n+1},y_{n+1})\coloneqq (\phi(s),\psi(t))\in\So\times\St$ with $f_{\alpha,\lambda_0}(x_{n+1},y_{n+1})\geq f_{\alpha,\lambda_0}(x_n,y_n)$.

Algorithm~\ref{alg:boundary} summarizes in pseudocode how we proceed towards the boundary in direction of the positive real axis. It takes the matrix $\A$ in form of its decompositon parts $A$, $B$, $C$ and $D$, a starting point $\lambda_0\in W^2(\A)$ with corresponding $(x_0,y_0)\in\So\times\St$, a penalty constant $p$, an angle $\alpha$ and a number of iterations $i_{\mathrm{max}}$ as its input and returns arrays $x\subset\So$ and $y\subset\St$. Note, that the algorithm does not only return the vectors associated to the point closest to the boundary after $i_{\mathrm{max}}$ iterations but an array of vectors that can be used to compute some points along the way. Later on, these points can be plotted as well in order to fill out the interior of the quadratic numerical range.

\begin{algorithm2e}[h]
    \SetAlgorithmStyle
    \caption{Seeking the Boundary}
    \label{alg:boundary}
    \KwIn{$A$, $B$, $C$, $D$, $x_0$, $y_0$, $\lambda_0$, $p$, $\alpha$, and $i_{\mathrm{max}}$}
    \Fun{$f(s,t,A,B,C,D,x,y,u,v,p,\alpha,\lambda_0)$}{
      $x = \cos(s)x + \sin(s)u$\;
      $y = \cos(t)y + \sin(t)v$\;
      \KwRet $\Re\lambda^{(\alpha)}_{x,y}-p\Im\big(\lambda^{(\alpha)}_{x,y}-\lambda_0\big)^2$\;
    }
    \Fun{$\mathrm{find\_boundary}(A,B,C,D,x,y,\lambda_0,p,\alpha)$}{
      \If{$2\lambda^{(\alpha)}_{x,y}\neq a_{x}+d_{y}$}{
        $S = \frac{1}{2\lambda^{(\alpha)}_{x,y}-a_x-d_y}\Big[ \begin{smallmatrix} (\lambda^{(\alpha)}_{x,y}-d_y)A & c_{x,y} B \\ b_{y,x} C & (\lambda^{(\alpha)}_{x,y}-a_x)D \end{smallmatrix} \Big]$\;
        $T_+ = S+S^*$\;
        $T_- = S-S^*$\;
        $T = T_+ + 2p\Im\big(\lambda^{(\alpha)}_{x,y}-\lambda_0\big)\iu T_-$\;
        $\big[\begin{smallmatrix}w\\z\end{smallmatrix}\big] =  T \big[\begin{smallmatrix}x\\y\end{smallmatrix}\big]$\;
        \If{$w\neq0$ and $z\neq0$}{
          $u = \frac{w-\Re\langle w,x\rangle x}{\norm{w-\Re\langle w,x\rangle x}}$\;
          $v = \frac{z - \Re\langle z,y\rangle y}{\norm{z - \Re\langle z,y\rangle y}}$\;
          Determine $(s,t)\in[0,2\pi]\times[0,2\pi]$ such that $f(s,t,A,B,C,D,x,y,u,v,p,\alpha,\lambda_0)$ is maximal\;
          $x = \cos(s)x + \sin(s)u$\;
          $y = \cos(t)y + \sin(t)v$\;
        }
      }
      \KwRet $\big(x,y)$\;
    }
    $(x[0],y[0]) = \mathrm{find\_boundary}(A,B,C,D,x_{0},y_{0},\lambda_{0},p,\alpha)$\;
    \For{$i=1,\dots,i_{\mathrm{max}}-1$}{
      $(x[i],y[i]) = \mathrm{find\_boundary}(A,B,C,D,x[i-1],y[i-1],\lambda_0,p,\alpha)$\;
      \If{$(x[i],y[i]) == (x[i-1],y[i-1])$}{
        \KwRet $\{(x[0],y[0]),\dots,(x[i-1],y[i-1])\}$\;
      }
    }
    \KwRet $(x,y)$\;
\end{algorithm2e}

\subsection{Box Approach}

In order to formulate an algorithm which computes the quadratic numerical range of a given matrix $\A$ with increasing quality we proceed as follows:\\
Initially, we fix an $\alpha\in[0,\pi[$ for the objective function \eqref{eq:objective_function} and compute a few points within $W^2(\A)$ by using the random vector sampling method, i.e.~we randomly generate some $(x,y)\in\So\times\St$ and insert the eigenvalues $\lambda^{(\alpha)}_{x,y}$ of $M_{x,y}$ in an array $W$ and the other eigenvalues $\lambda^{(\alpha+\pi)}_{x,y}$ of $M_{x,y}$ in a second array $\widetilde{W}$. Those points will serve as candidates for the starting points $\lambda_0$ of Algorithm~\ref{alg:boundary}, but in order to control their number and decrease the required iteration steps of Algorithm~\ref{alg:boundary}, we will preselect the starting points via a box approach such that the computational cost will be reduced.

We start by creating a rectangular grid of small equally sized boxes covering all the sampled points in $W$. Then, we pick one sample point from the interior of each box that is non-empty as a representative and determine all non-empty boxes that are not surrounded by other non-empty boxes. Now, only the representatives of those boxes will be used as a starting point $\lambda_0$. This ensures that the $\lambda_0$ will be evenly spread throughout the cloud of computed points even though they might cluster. Moreover, it allows us to restrict our choice of starting points to those that are presumably already close to the boundary, which leads to a higher chance of reaching the boundary within only a few iterations of Algorithm~\ref{alg:boundary}.

For each starting point, we will then select some randomly oriented but equally spread search directions, rotate the matrix $\A$ accordingly and proceed towards the boundary via Algorithm~\ref{alg:boundary}. This yields new vectors $(x,y)\in\So\times\St$ and we subsequently insert the new corresponding points $\lambda^{(\alpha)}_{x,y}$ into $W$. At this point, it has also shown to be advantageous to compute the other eigenvalues as well and insert them into $\widetilde{W}$.

Afterwards, we do the same for ${\alpha+\pi}$ in place of $\alpha$ and $\widetilde{W}$ interchanged with $W$ by using a separate grid of boxes and obtain larger clouds of points $W$ and $\widetilde{W}$ as a result. Now, the whole process can be repeated via the construction of new grids of boxes.

If we repeat this procedure over and over again, the number of starting points will eventually remain relatively constant. At this point, we increase the number of boxes, i.e.~reduce their size, in order to increase the resolution of the box approach and therefore heighten our ability to distinguish potential starting points in the interior from potential starting points close to the boundary.\\

\begin{minipage}{0.23\textwidth}
\begin{tikzpicture}[every node/.style={minimum size=1cm-\pgflinewidth, outer sep=0pt}]
  \colorlet{shadecolor}{gray!50}
  \node[fill=shadecolor] at (0.5,0.5) {};
  \node[fill=shadecolor] at (0.5,1.5) {};
  \node[fill=shadecolor] at (0.5,2.5) {};
  \node[fill=shadecolor] at (1.5,0.5) {};
  \node[fill=gray] at (1.5,1.5) {};
  \node[fill=gray] at (1.5,2.5) {};
  \node[fill=shadecolor] at (2.5,0.5) {};
  \node[fill=shadecolor] at (2.5,1.5) {};
  \node[fill=shadecolor] at (2.5,2.5) {};
  \node[fill=shadecolor] at (3.5,2.5) {};
  \node[fill=shadecolor] at (2.5,3.5) {};
  \node[fill=shadecolor] at (0.5,3.5) {};
  \node[fill=shadecolor] at (1.5,3.5) {};
  \draw [step=1.0,black, very thick] (0,0) grid (4,4);
  \draw [fill] (4,2.6) circle [radius=2pt];
  \draw  (3.4,2.7) circle [radius=5pt];
  \draw [fill] (0.4,0) circle [radius=2pt];y
  \draw  (0.4,0) circle [radius=5pt];
  \draw [fill] (0.8,0.85) circle [radius=2pt];
  \draw [fill] (0.6,1.3) circle [radius=2pt];
  \draw [fill] (0.6,1.7) circle [radius=2pt];
  \draw  (0.6,1.3) circle [radius=5pt];
  \draw [fill] (0,2.2) circle [radius=2pt];
  \draw  (0.5,2.7) circle [radius=5pt];
  \draw [fill] (0.5,2.7) circle [radius=2pt];
  \draw [fill] (1.5,0.7) circle [radius=2pt];
  \draw  (1.5,0.7) circle [radius=5pt];
  \draw [fill] (1.7,0.2) circle [radius=2pt];
  \draw [fill] (1.4,1.2) circle [radius=2pt];
  \draw [fill] (1.8,1.5) circle [radius=2pt];
  \draw  (1.8,1.5) circle [radius=5pt];
  \draw [fill] (1.7,1.8) circle [radius=2pt];
  \draw [fill] (1.6,2.6) circle [radius=2pt];
  \draw  (1.6,2.6) circle [radius=5pt];
  \draw [fill] (2.5,0.4) circle [radius=2pt];
  \draw  (2.5,0.4) circle [radius=5pt];
  \draw [fill] (2.3,0.6) circle [radius=2pt];
  \draw [fill] (2.6,1.6) circle [radius=2pt];
  \draw  (2.6,1.6) circle [radius=5pt];
  \draw [fill] (2.2,2.2) circle [radius=2pt];
  \draw [fill] (2.6,2.2) circle [radius=2pt];
  \draw  (2.6,2.2) circle [radius=5pt];
  \draw [fill] (2.5,2.8) circle [radius=2pt];
  \draw [fill] (3.4,2.7) circle [radius=2pt];
  \draw [fill] (2.1,3.5) circle [radius=2pt];
  \draw  (2.3,4) circle [radius=5pt];
  \draw [fill] (2.3,4) circle [radius=2pt];
  \draw [fill] (0.2,3.2) circle [radius=2pt];
  \draw [fill] (0.7,3.5) circle [radius=2pt];
  \draw  (0.7,3.5) circle [radius=5pt];
  \draw [fill] (1.5,3.2) circle [radius=2pt];
  \draw  (1.5,3.2) circle [radius=5pt];
  \draw [fill] (1.3,2.4) circle [radius=2pt];
\end{tikzpicture}
\captionof{figure}{Grid of boxes}
\label{fig:Boxes}
\end{minipage}
\hspace{.1\textwidth}
\begin{minipage}{0.61\textwidth}
  \Cref{fig:Boxes} illustrates the box approach in a small example. Here, the black dots represent the cloud of points in $W$ that have been computed so far and a rectangular grid of boxes is constructed covering all points. In every non-empty box (gray), one of these points is selected and surrounded by a circle. Now, only the circled points in the light gray boxes will be used as starting points in the next iteration since the dark gray boxes are surrounded by other non-empty boxes.\\
  \vfill
\end{minipage}

\medskip

When it comes to the determination of the penalty constant $p$ one has to balance two aspects: Larger penalty constants lead to higher accuracy in the given search direction, c.f. \Cref{thm:local_max}, while smaller penalty constants result in a faster convergence towards the boundary. We therefore start with a small penalty constant to cover a large area in the beginning and increase $p$ over time while also making it dependent on the diameter of the cloud of points in the current iteration.

Algorithm~\ref{alg:grid} explains in pseudocode, how the starting points are selected and how $p$ is chosen. It takes arrays $x\subset\So$, $y\subset\St$ and $W\subset W^2(\A)$, the square root of the number of boxes $\mathpzc{B}$ and the current iteration $i$ as its input and returns arrays $x_0\subset\So$, $y_0\subset\St$ and $\lambda_0\subset W^2(\A)$ as well as the penalty constant $p$.\\

\begin{algorithm2e}[h]
    \SetAlgorithmStyle
    \caption{Grid}
    \label{alg:grid}
    \KwIn{$x$, $y$, $W$, $\mathpzc{B}$ and $i$}
    $x_0 = \{\}$, $y_0 = \{\}$, $\lambda_0 = \{\}$, $G = \mathrm{zeros}(\mathpzc{B},\mathpzc{B})$ and $I = \mathrm{zeros}(\mathpzc{B},\mathpzc{B})$\;
      $\Re_{\max} = \max\Re W$, $\Im_{\max} = \max\Im W$, $\Re_{\min} = \min\Re W$ and $\Im_{\min} = \min\Im W$\;
      $p = \frac{i^2}{100}\max\{|\Re_{\max} -\Re_{\min}|, |\Im_{\max} -\Im_{\min} |\}$\;
      $h_\Re = (\Re_{\max}-\Re_{\min})/\mathpzc{B}$ and $h_\Im = (\Im_{\max}-\Im_{\min})/\mathpzc{B}$\;
      \For{$j=0$ to the length of $W$ $-1$}{
        $k = \mathrm{integer}((\Im_{\max}-\Im W[j])/h_\Im)$ and $l = \mathrm{integer}((\Re W[j]-\Re_{\min})/h_\Re)$\;
        \textbf{if} {$k == \mathpzc{B}$} \textbf{then} {$k \mathrel{-}= 1$} and \textbf{if} {$l == \mathpzc{B}$} \textbf{then} {$l \mathrel{-}= 1$}\;
        \lIf{$G[k][l] == 0$}{
          $G[k][l] = 1$ and $I[k][l] = j+1$
        }
      }
      \For{$k=1,\dots,\mathpzc{B}-1$}{
        \For{$l=1,\dots,\mathpzc{B}-1$}{
          \lIf{$G[k+m][l+n] = 1$ for all $m,n\in\{-1,0,1\}$}{
            $I[k][l] = 0$
          }
        }
      }
      $j = 0$\;
      \For{$k=1,\dots,\mathpzc{B}-1$}{
        \For{$l=1,\dots,\mathpzc{B}-1$}{
          \If{$I[k][l]\neq 0$}{
            $x_0[j] = x[I[k][l]-1]$, $y_0[j] = y[I[k][l]-1]$, $\lambda_0[j] = W[I[k][l]-1]$ and $j\mathrel{+}=1$\;
          }
        }
      }
      \KwRet $(x_0,y_0,\lambda_0,p)$\;
    \end{algorithm2e}

Algorithm~\ref{alg:boxes} contains the full instructions for the computation of the numerical range.
It takes the matrix $\A$ in form of its decomposition parts $A$, $B$, $C$ and $D$, an angle $\alpha$ and the time the algorithm should run for $\tau_{\mathrm{max}}$ as its input and returns a cloud of points $(W,\widetilde{W})\subset W^2(\A)$. The square roots of the numbers of boxes $\mathpzc{B}$ and $\tilde{\mathpzc{B}}$ will be increased in dependence of counters $\mathpzc{c}$ and $\tilde{\mathpzc{c}}$ that keep track of the number of starting points.

\begin{algorithm2e}[h!]
    \SetAlgorithmStyle
    \caption{Computing the Quadratic Numerical Range}
    \label{alg:boxes}
    \KwIn{$A$, $B$, $C$, $D$, $\alpha$ and $\tau_{\mathrm{max}}$}
    $\tau = \text{current time} + \tau_{\mathrm{max}}$ and $\mathrm{timeflag} = \mathrm{false}$\;
    $\A = \big[\begin{smallmatrix}A & B \\ C & D\end{smallmatrix}\big]$, $W = \{\}$, $\widetilde{W} = \{\}$, $\mathpzc{B} = 20$, $\tilde{\mathpzc{B}} = 20$, $\mathpzc{c} = 0$ and $\tilde{\mathpzc{c}} = 0$\;
    Generate arrays of random vectors $x\subset \So$ and $y\subset\St$\;
    $n = \text{length of } x$\;
    \For{$i=0$ to $n-1$}{
      $W[i] = \lambda^{(\alpha)}_{x[i],y[i]}$ and $\widetilde{W}[i] = \lambda^{(\alpha+\pi)}_{x[i],y[i]}$\;
    }
    \For{$i=0,\dots,\infty$}{
      \For{$j = 0, \pi$}{
        \lIf{$j == 0$}{$(x_0,y_0,\lambda_0,p) = \mathrm{Grid}(x,y,W,\mathpzc{B},i)$}
        \lIf{$j == \pi$}{$(x_0,y_0,\lambda_0,p) = \mathrm{Grid}(x,y,\widetilde{W},\tilde{\mathpzc{B}},i)$}
        \For{$k=0$ to the length of $\lambda_0$ $-1$}{
          $\theta_0 = $ random angle in $[0,2\pi[$\;
          \For{$l=0,\dots,4$}{
            $\theta = \theta_0 + l\frac{2\pi}{5}$ and $\big[\begin{smallmatrix}A & B \\ C & D\end{smallmatrix}\big] = \e^{\iu\theta}\A$\;
            $(\hat{x},\hat{y}) = \mathrm{Seeking\_the\_Boundary}(A,B,C,D,x_0[k],y_0[k],\e^{\iu\theta}\lambda_0[k],p,\alpha+j-\theta,2)$\;
            $\hat{n} = \text{length of } \hat{x}$\;
            \For{$m=0$ to $\hat{n}-1$}{
              $W[n+m] = \lambda^{(\alpha)}_{\hat{x}[m],\hat{y}[m]}$ and $\widetilde{W}[n+m] = \lambda^{(\alpha+\pi)}_{\hat{x}[m],\hat{y}[m]}$\;
              $x[n+m] = \hat{x}[m]$ and $y[n+m] = \hat{y}[m]$\;
            }
            \lIf{current time $>\tau$}{$\mathrm{timeflag} = \mathrm{true}$ and \textbf{break}}
            $n \mathrel{+}= \hat{n}$\;
          }
          \lIf{$\mathrm{timeflag} == \mathrm{true}$}{\textbf{break}}
        }
        \lIf{$\mathrm{timeflag} == \mathrm{true}$}{\textbf{break}}
        \If{$j == 0$}{
          \lIf{$1 \geq (\text{length of }\lambda_0)/\mathpzc{c} > 0.99$}{$\mathpzc{B} = \mathrm{integer}(\sqrt{2}\mathpzc{B})$}
          $\mathpzc{c} = \text{ length of }\lambda_0$\;
        }
        \If{$j == \pi$}{
          \lIf{$1 \geq (\text{length of }\lambda_0)/\tilde{\mathpzc{c}} > 0.99$}{$\tilde{\mathpzc{B}} = \mathrm{integer}(\sqrt{2}\tilde{\mathpzc{B}})$}
          $\tilde{\mathpzc{c}} = \text{ length of }\lambda_0$\;
        }
      }
      \lIf{$\mathrm{timeflag} == \mathrm{true}$}{\textbf{break}}
    }
    \KwRet $(W,\widetilde{W})$\;
\end{algorithm2e}

\section{Examples}\label{sect:examples}

The following pictures are the result of a Python implementation of Algorithm~\ref{alg:boxes}. Here, the search for $(s,t)\in[0,2\pi]\times[0,2\pi]$ such that $f_{\alpha,\lambda_0}\big(\phi(s),\psi(t)\big)$ is maximal with $\varphi$ and $\psi$ as in \eqref{eq:phi_psi} is further reduced to a one-dimensional optimization, i.e.~$s=t$, in order to speed this step up. This allows us to compute much more points in the QNR within the same amount of time and we obtain better pictures in the end. 

\begin{example}
  Let us consider the matrix
\begin{equation*}
  \A_{1} = \left[\begin{array}{ccccc|ccccc}
                 2 & \iu & 0 & \dots & 0 & 1 & 3+\iu & 0 & \dots & 0\\
                 \iu & \ddots & \ddots & \ddots & \vdots & 3+\iu & \ddots & \ddots & \ddots & \vdots\\
                 0 & \ddots & \ddots & \ddots & 0  & 0 & \ddots  & \ddots & \ddots & 0\\
                 \vdots & \ddots & \ddots  & \ddots & \iu  & \vdots & \ddots   &\ddots  & \ddots & 3+\iu\\
                 0 & \dots & 0 & \iu & 2 & 0 & \dots & 0 & 3+\iu & 1\\
                 \hline
                 1 & 3+\iu & 0 & \dots & 0 & -2 & \iu & 0 & \dots & 0\\
                 3+\iu & \ddots & \ddots & \ddots & \vdots & \iu & \ddots & \ddots & \ddots & \vdots\\
                 0 & \ddots & \ddots & \ddots & 0  & 0 & \ddots  & \ddots & \ddots & 0\\
                 \vdots & \ddots & \ddots  & \ddots & 3+\iu  & \vdots & \ddots   &\ddots  & \ddots & \iu\\
                 0 & \dots & 0 & 3+\iu & 1 & 0 & \dots & 0 & \iu & -2  \end{array}\right],
\end{equation*}
where the blocks $A$, $B$, $C$ and $D$ are equally sized tridiagonal matrices, cf. \cite[Example 1.1.5]{Tretter}. In \Cref{fig:A2}, the results of the algorithm are compared to the random vector sampling method when executed for the determination of the QNR of $\A_{1}$ with $\dim(\A_1) = 40$. In the top row, the execution time was one minute and in the bottom row one hour while the plots in the left column are a result of the algorithm and the plots in the right column are a result of the random vector sampling method. Here, $\alpha$ was chosen to be zero such that the sets $W$ (dark blue) and $\widetilde{W}$ (light blue) coincide with the two disconnected components of $W^{2}(\A_{1})$. As we see, the algorithm is capable of detecting the rough shape of the quadratic numerical range already after a short period of time and refines the result very well afterwards while the random vector sampling method is only able to locate a small area of the QNR which gets slightly enlarged over time.

\Cref{fig:A2_large} demonstrates the efficacy of the algorithm (left) even when the dimension of $\A_1$ is increased to $4000$. Here, the superiority over the random vector sampling method (right) becomes even more visible. The computation time was two hours in both pictures.
\begin{figure}[h]
\includegraphics[width=\textwidth]{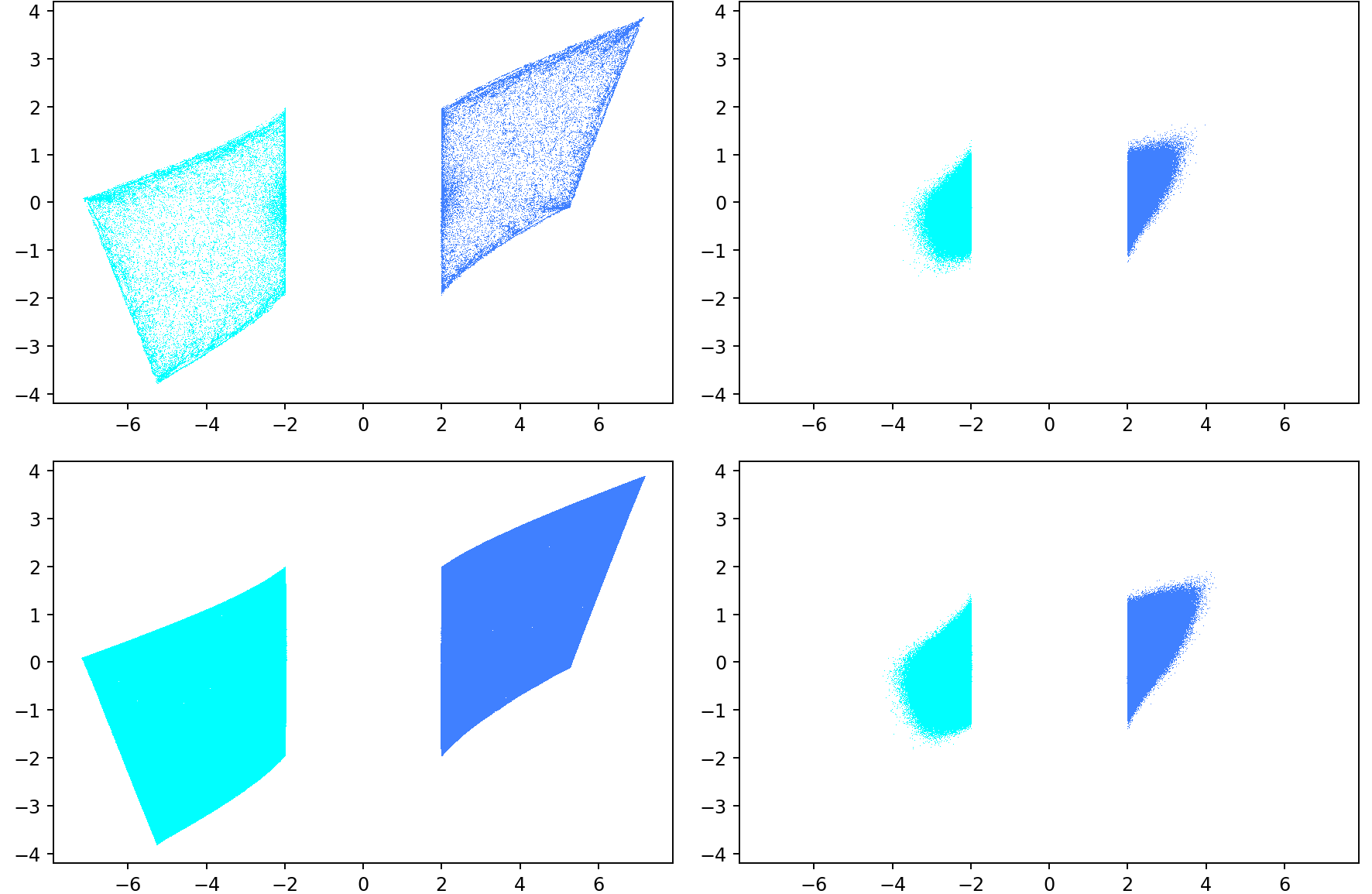}
\captionof{figure}{QNR of $\A_{1}$: Algorithm versus random vector sampling for different amounts of time}
\label{fig:A2}
\end{figure}

\bigskip

\begin{figure}[h!]
\includegraphics[width=\textwidth]{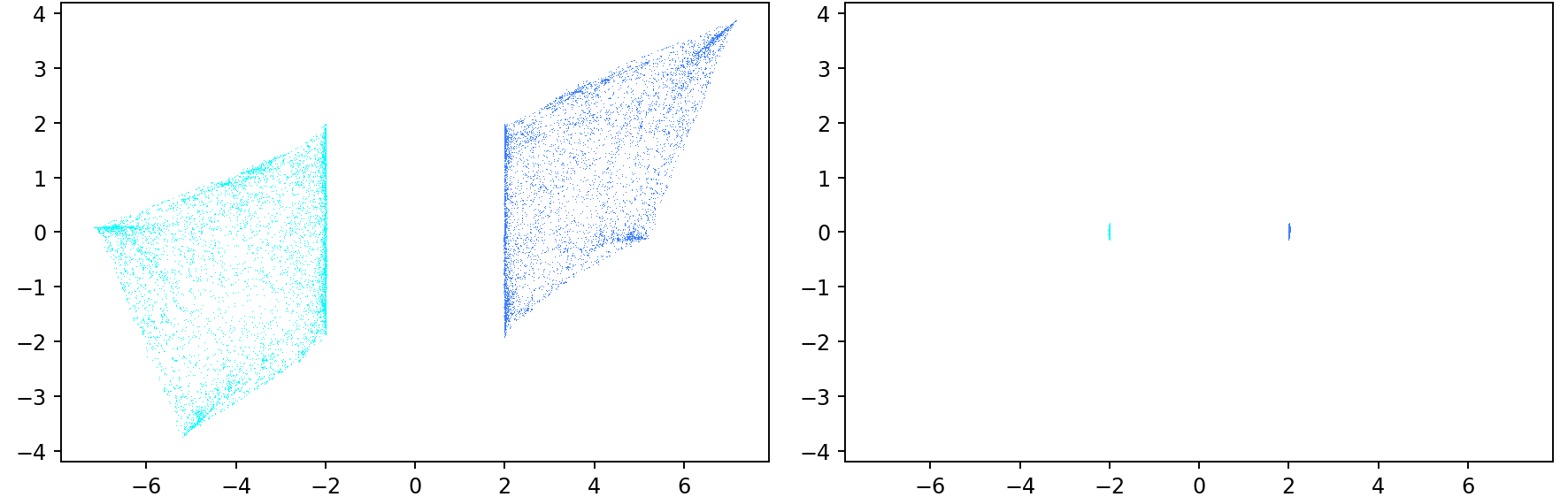}
\captionof{figure}{QNR of high dimensional $\A_{1}$: Algorithm versus random vector sampling}
\label{fig:A2_large}
\end{figure}
\end{example}

\begin{example}
  Let us consider a smaller matrix like
  \begin{equation*}
\A_{2} = \left[ \begin{array}{cccc|cccc} 0 & 0 & 0 & 0 & 1 & 0 & 0 & 0 \\ 0 & 0 & 0 & 0 & 0 & 1 & 0 & 0 \\ 0 & 0 & 0 & 0 & 0 & 0 & 1 & 0 \\ 0 & 0 & 0 & 0 & 0 & 0 & 0 & 1 \\ \hline -2 & -1 & 0 & 0 & \iu & 5\iu & 0 & 0 \\ -1 & -2 & -1 & 0 & -5\iu & \iu & 5\iu & 0 \\ 0 & -1 & -2 & -1 & 0 & -5\iu & \iu & 5\iu \\ 0 & 0 & -1 & -2 & 0 & 0 & -5\iu & \iu \end{array}\right],
\end{equation*}
cf. \cite[Example 1.3.3]{Tretter}. \Cref{fig:Hat} shows the quadratic numerical range of $A_{2}$ after executing the algorithm and the random vector sampling method with $\alpha = \pi/2$ for one hour each. Although the random vector method generally covers a much bigger part of the quadratic numerical range of smaller matrices like this when compared to higher dimensional ones, here, it still fails to adequately depict some parts of the boundary and struggles to close the gap between the two components, which seem to be connected as the plot of the algorithm suggests. As we see, this is not compensated by the fact that only $3.357.980$ points were computed with the algorithm while $124.581.432$ points were sampled via the random vector method within the same amount of time.
\begin{figure}[h]
	\includegraphics[width=\textwidth]{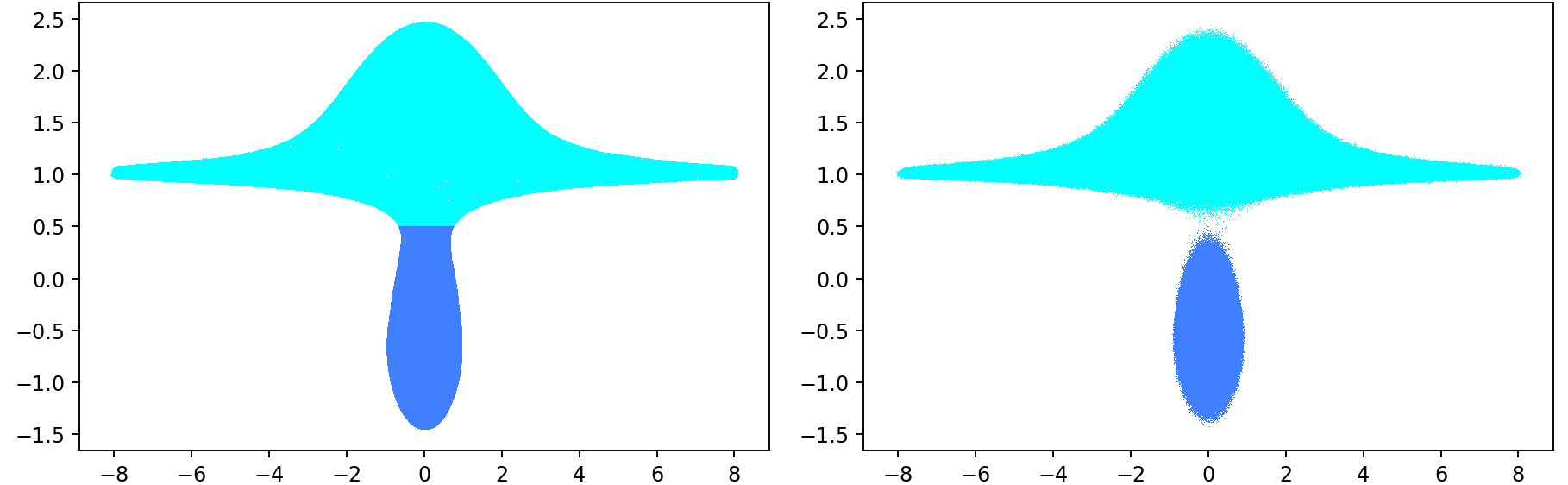}
	\captionof{figure}{QNR of $\A_{2}$ with the algorithm (left) and random vector sampling (right)}
	\label{fig:Hat}
\end{figure}
\end{example}

\begin{example}
  Let us consider the matrices
  \begin{equation*}
\A_{3} = \left[\begin{array}{cc|cc}0 & 0 & 0 & 1 \\ 0 & 1 & 2 & 3 \\ \hline 0 & -2 & -1 & 0 \\ -1 & -3 & 0 & 0 \end{array}\right],
\end{equation*}
cf. \cite[Example 1.3.10]{Tretter}, and
\begin{equation*}
	\A_{4} = \left[\begin{array}{ccc|cc}0 & 0 & 0 & 0 & 0 \\ 0 & 0 & 1+\iu  & 0 & 0 \\ 0 & 2\iu & 0 & 0 & 0 \\ \hline 0 & 0 & 0 & 0 & 0 \\ -1 & 2 & -2 & \iu & 0 \end{array}\right].
\end{equation*}
\begin{minipage}{0.5\textwidth}
	\includegraphics[width=\textwidth]{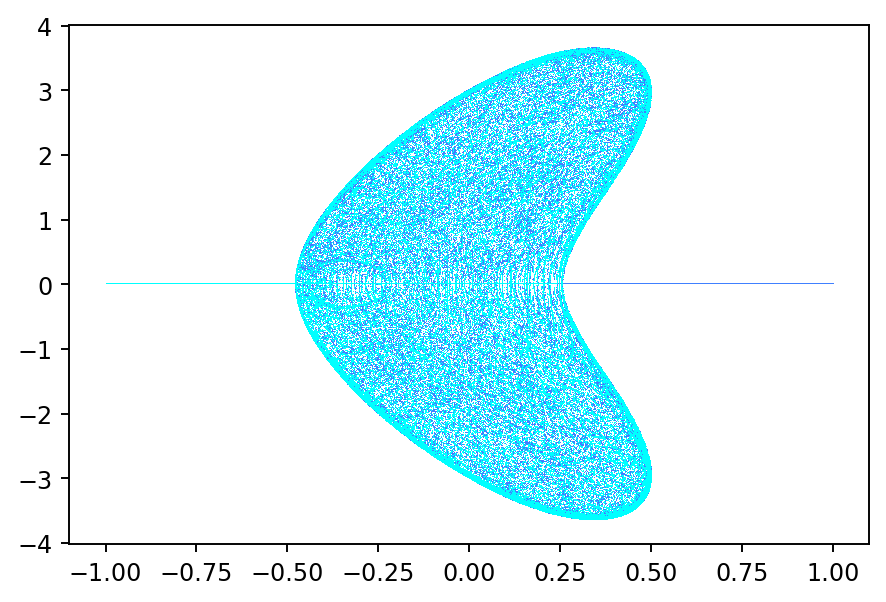}
	\captionof{figure}{QNR of $\A_{3}$}
	\label{fig:A3}
\end{minipage}
\begin{minipage}{0.5\textwidth}
	\includegraphics[width=\textwidth]{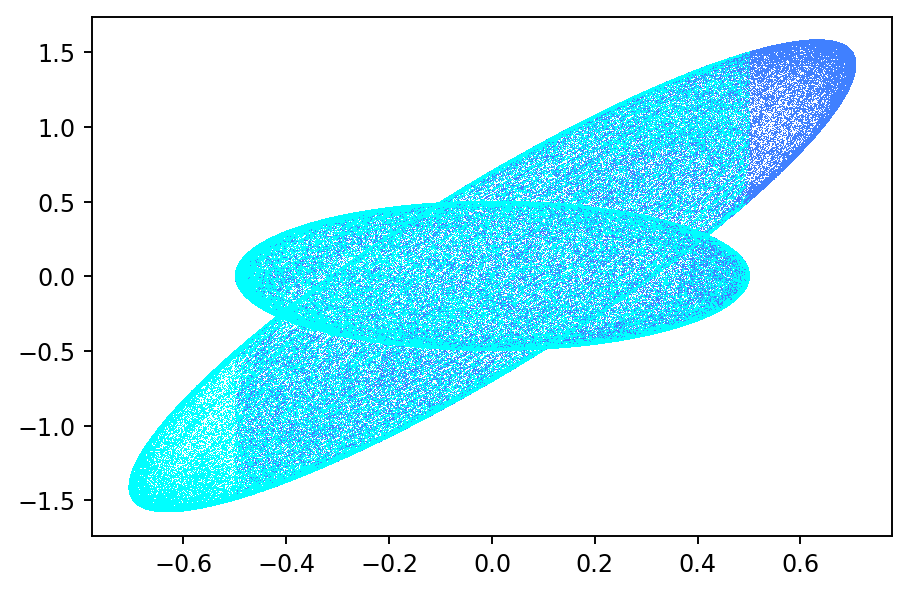}
	\captionof{figure}{QNR of $\A_{4}$}
	\label{fig:X}
\end{minipage}\\

\Cref{fig:A3,fig:X} demonstrate how the result of the algorithm can look like if the QNR consists of only one connected component and $\alpha$ is arbitrarily chosen to be $0$. They are the result of an execution of the algorithm for 15 minutes each.
\end{example}

\section{Concentration Phenomenon for Random Sampling}\label{sect:concentration}

As we see in the examples of \Cref{sect:examples} and especially in \Cref{fig:A2,fig:A2_large} the points in the quadratic numerical range computed via the random vector sampling method are very unequally spread and cluster in a small subset of each component. In this section, we will examine this phenomenon and prove that the probability of a sampling point to fall outside of a small neighborhood of the expected value decays exponentially with an increase of the dimension of the matrix when its norm stays constant.

\begin{prop}\label{thm:EW}
  Consider the probability spaces $(S_{\H_{i}},\mathcal{B}(S_{\H_{i}}),\sigma_{i})$, $i=1,2$, where $\sigma_{i}$ is the normalized surface measure. Let $M\colon\So\times\St \to \C^{2\times 2}$, $(x,y)\mapsto \begin{bmatrix} \langle Ax,x\rangle & \langle By,x\rangle \\ \langle Cx,y\rangle & \langle Dy,y\rangle \end{bmatrix}$. Then the expected value $\E M$ of $M$ is given by
  \begin{equation*}
    \E M = \begin{bmatrix} \frac{\trace(A)}{\dim(\H_1)} & 0 \\ 0 & \frac{\trace(D)}{\dim(\H_2)} \end{bmatrix}.
  \end{equation*}
\end{prop}

\begin{proof}
  The expected value is
  \begin{align*}
    \E M &= \int_{\So\times\St} \begin{bmatrix} \langle Ax,x\rangle & \langle By,x\rangle \\ \langle Cx,y\rangle & \langle Dy,y\rangle \end{bmatrix} \, \dd (\sigma_1\times\sigma_2)(x,y)\\
    &= \begin{bmatrix} \int_{\So}\langle Ax,x\rangle\,\dd\sigma_1(x) & \int_{\So\times\St}\langle By,x\rangle \, \dd(\sigma_{1}\times\sigma_{2})(x,y) \\ \int_{\So\times\St}\langle Cx,y\rangle\,\dd(\sigma_{1}\times\sigma_{2})(x,y) & \int_{\St}\langle Dy,y\rangle\,\dd \sigma_2(y) \end{bmatrix},
  \end{align*}
  where
  \begin{align*}
    \int_{\So\times\St}\langle By,x\rangle\,\dd(\sigma_{1}\times\sigma_{2})(x,y) &= \int_{\St}\langle By,\int_{\So}x\,\dd\sigma_{1}(x)\rangle\,\dd\sigma_{2}(y)\\
                                                                          &= \int_{\St}\langle By,0\rangle\,\dd\sigma_{2}(y)\\
    &= 0
  \end{align*}
  by Fubini's theorem and via a similar argumentation we also obtain
  \begin{equation*}
    \int_{\So\times\St}\langle Cx,y\rangle\,\dd(\sigma_{1}\times\sigma_{2})(x,y) = 0.
  \end{equation*}
  Let $d$ be the dimension of $\H_1$ and denote by $e_1,\dots,e_d$ an orthonormal basis of $\H_1$. Then the trace of $A$ is given by
  \begin{align*}
    \trace(A) &=  \sum_{k=1}^d \langle Ae_k,e_k\rangle \\
              &= \sum_{k=1}^d \langle AUe_k,Ue_k\rangle,
  \end{align*}
  where $U$ is an arbitrary unitary matrix. Denoting the normalized Haar measure on the unitary group $\U(d)$ by $\mu$ we have by Fubini's theorem
  \begin{align*}
    \trace(A) &= \int_{\U(d)}\sum_{k=1}^d \langle AUe_k,Ue_k\rangle\,\dd\mu(U) \\
              &= \int_{\So}\int_{\U(d)}\sum_{k=1}^d \langle AUe_k,Ue_k\rangle\,\dd\mu(U)\,\dd\sigma_1(x) \\
              &= d\int_{\So}\int_{\U(d)}\langle AUx,Ux\rangle\,\dd\mu(U)\,\dd\sigma_1(x)\\
              &= d\int_{\U(d)}\int_{\So}\langle AUx,Ux\rangle\,\dd\sigma_1(x)\,\dd\mu(U) \\
              &= d\int_{\So}\langle Ax,x\rangle\,\dd\sigma_1(x)
  \end{align*}
  due to the invariance of the Haar measure and $\sigma_1$ under unitary transformations.
  
  It follows analogously that
  \begin{equation*}
     \trace(D) = \dim(\H_2)\int_{\St}\langle Dy,y\rangle\,\dd\sigma_2(y),
  \end{equation*}
  which concludes the proof.
\end{proof}

\begin{defi}
  For two nonempty sets $K,L\subset \C$ we define the distance $\dist(K,L)$ via
  \begin{equation*}
    \dist(K,L) = \sup_{k\in K}(\inf_{l\in L}\norm{k-l})
  \end{equation*}
  and the Hausdorff-distance $\dH(K,L)$ via
  \begin{equation*}
    \dH(K,L) = \max\{\dist(K,L),\dist(L,K)\}.
  \end{equation*}
\end{defi}

Note, that $K\subset B_\eps(L)$ whenever $\dist(K,L)<\eps$ and $\dist(K,L)\leq\eps$ whenever $K\subset B_\eps(L)$.

\begin{lemma}\label{thm:M1M2}
  Let $M_{1},M_{2}\in\C^{2\times 2}$. Then
  \begin{equation*}
    \dH\big(\sigma(M_1),\sigma(M_2)\big) \leq \big((\norm{M_1}+\norm{M_2})\norm{M_1-M_2}\big)^{\frac{1}{2}}.
  \end{equation*}
\end{lemma}

\begin{proof}
  For a $\lambda\in\varrho(M_1)$, \cite[p.~28]{Kato95} yields that
  \begin{equation*}
    \norm{(\lambda-M_{1})^{-1}} \leq \frac{\norm{\lambda-M_{1}}}{\left|\det(\lambda-M_{1})\right|} \leq \frac{\norm{\lambda-M_{1}}}{\dist(\lambda,\sigma(M_{1}))^2}
  \end{equation*}
  or in other words
  \begin{equation*}\label{eq:otherwords}
    \dist(\lambda,\sigma(M_{1})) \leq \left(\norm{\lambda-M_{1}}\norm{(\lambda-M_{1})^{-1}}^{-1}\right)^{\frac{1}{2}}.
  \end{equation*}
  If we further assume $\lambda\in\sigma(M_{2})$, we have
  \begin{equation*}
    \norm{\lambda-M_1} \leq \norm{M_2} + \norm{M_1}
  \end{equation*}
  on one hand and on the other hand we obtain
  \begin{equation*}
    \norm{(\lambda-M_{1})^{-1}}^{-1} \leq \norm{M_{2}-M_{1}}
  \end{equation*}
  because otherwise $\norm{(\lambda- M_{1})^{-1}}^{-1} > \norm{M_{2}-M_{1}}$ implies
  \begin{equation*}
    \norm{(M_{2}-M_{1})(\lambda-M_{1})^{-1}}<1
  \end{equation*}
  and therefore $I-(M_{2}-M_{1})(\lambda-M_{1})^{-1} = (\lambda-M_{2})(\lambda-M_{1})^{-1}$ is invertible by a Neumann series argument. This yields $\lambda\in\varrho(M_{2})$, which is a contradiction.

  Hence, we have
  \begin{equation*}
    \dist(\lambda,\sigma(M_{1})) \leq \big((\norm{M_{1}}+\norm{M_2})\norm{M_{1}-M_{2}}\big)^{\frac{1}{2}}
  \end{equation*}
  for every $\lambda\in\sigma(M_2)$ and we analogously obtain
  \begin{equation*}
    \dist(\lambda,\sigma(M_{2})) \leq \big((\norm{M_{1}}+\norm{M_2})\norm{M_{1}-M_{2}}\big)^{\frac{1}{2}}
  \end{equation*}
  for every $\lambda\in\sigma(M_1)$. Thus,
  \begin{align*}
    \dH\big(\sigma(M_1),\sigma(M_2)\big) &= \max\left\{\dist\big(\sigma(M_1),\sigma(M_2)\big),\dist\big(\sigma(M_2),\sigma(M_1)\big)\right\}\\
                                         &= \max\left\{\sup_{\lambda\in\sigma(M_1)}\dist(\lambda,\sigma(M_2)),\sup_{\lambda\in\sigma(M_2)}\dist(\lambda,\sigma(M_1))\right\}\\
                                         &\leq \big((\norm{M_1}+\norm{M_2})\norm{M_1-M_2}\big)^{\frac{1}{2}}. \qedhere
  \end{align*}
\end{proof}

\begin{theo}\label{thm:concentration}
  Denote by $S^{n-1}$ the $(n-1)$-dimensional sphere in $\R^n$ and let $f\colon S^{n-1}\to\R$ be a function with Lipschitz constant $L$. Then for all $\eps>0$ we have
  \begin{equation*}
    \sigma\left(\left| f(x)-\int_{S^{n-1}} f\,\dd\sigma\right| > \eps\right) \leq 4\exp\left(-\frac{\delta\eps^2n}{L^2}\right)
  \end{equation*}
  where $\sigma$ is the normalized surface measure on $S^{n-1}$ and $\delta>0$ an absolute constant.
\end{theo}

\begin{proof}
\cite[Corollary V.2]{MilmanSchechtman}.
\end{proof}

\begin{theo}\label{thm:probability}
  Consider $M$ as in \Cref{thm:EW}.
  Then we have for all $\eps>0$
  \begin{equation*}
    \sigma_1\times\sigma_2\big( \dH(\sigma(M_{x,y}),\sigma(\E M)) > \eps\big) \leq 32\exp\left(-\beta \frac{\eps^4n_0}{\norm{\A}^4}\right)
  \end{equation*}
  where $\beta>0$ is an absolute constant and $n_0=\min\{\dim(\H_1),\dim(\H_2)\}$.
\end{theo}

\begin{proof}
  We start by considering the function $\Re\langle A\cdot, \cdot\rangle\colon\So\to\R$ for which we have
  \begin{equation*}
    |\Re\langle Ax,x\rangle - \Re\langle Ay,y\rangle| \leq 2\norm{A}\norm{x-y}, \quad x,y\in\So.
  \end{equation*}
  Thus, from \Cref{thm:concentration} and \Cref{thm:EW}, we obtain
  \begin{equation*}
    \sigma_1\left(\left|\Re\langle Ax,x\rangle - \Re \frac{\trace(A)}{\dim(\H_1)}\right|>\eps\right) \leq 4\exp\left({-\delta \eps^2\frac{\dim(\H_1)}{4\norm{A}^2}}\right)
  \end{equation*}
  and analogously
  \begin{equation*}
    \sigma_1\left(\left|\Im\langle Ax,x\rangle - \Im \frac{\trace(A)}{\dim(\H_1)}\right|>\eps\right) \leq 4\exp\left({-\delta \eps^2\frac{\dim(\H_1)}{4\norm{A}^2}}\right)
  \end{equation*}
  with an absolute constant $\delta>0$. Combining both of these statements, we get
  \begin{align}\label{eq:estA}
    \begin{split}
      \sigma_1\left(\left|\langle Ax,x\rangle - \frac{\trace(A)}{\dim(\H_1)}\right|>\eps\right) &\leq \sigma_1\left(\left|\Re\langle Ax,x\rangle - \Re \frac{\trace(A)}{\dim(\H_1)}\right|>\frac{\eps}{\sqrt{2}}\right)\\
      & \quad  + \sigma_1\left(\left|\Im\langle Ax,x\rangle - \Im \frac{\trace(A)}{\dim(\H_1)}\right|>\frac{\eps}{\sqrt{2}}\right)\\
      &\leq 8\exp\left({-\delta \eps^2\frac{\dim(\H_1)}{8\norm{A}^2}}\right)
      \end{split}
  \end{align}
  and via the same argumentation with $D$ in place of $A$ we obtain
  \begin{equation}\label{eq:estD}
    \sigma_2\left(\left|\langle Dy,y\rangle - \frac{\trace(D)}{\dim(\H_2)}\right|>\eps\right) \leq 8\exp\left({-\delta \eps^2\frac{\dim(\H_2)}{8\norm{D}^2}}\right).
  \end{equation}
  In order to find an estimate like this for $\sigma_1\times\sigma_2(|\langle By,x\rangle|>\eps)$ as well we first fix $y\in\St$ and consider the function $g_y\colon\So\to\C$, $x\mapsto \langle By,x\rangle$. For this we have
  \begin{equation*}
    \sigma_1\left(\left|g_y(x)\right|>\eps\right) \leq 8\exp\left({-\delta \eps^2\frac{\dim(\H_1)}{2\norm{B}^2}}\right)
  \end{equation*}
  again by \Cref{thm:concentration} and a similar argumentation as before because $\E g_y = 0$. Considering $\Omega\coloneqq \{(x,y)\in \So\times\St \mid |\langle By,x\rangle|>\eps\}$, we then obtain
  \begin{align}\label{eq:estB}
    \begin{split}
      \sigma_1\times\sigma_2(|\langle By,x\rangle|>\eps) &= \int_{\St}\int_{\So} \mathbbm{1}_\Omega(x,y)\,\dd\sigma_1\,\dd\sigma_2\\
      &= \int_{\St}\sigma_1\left(\left|g_y(x)\right|>\eps\right)\,\dd\sigma_2\\
    &\leq \int_{\St} 8\exp\left({-\delta \eps^2\frac{\dim(\H_1)}{2\norm{B}^2}}\right) \,\dd\sigma_2\\
    &= 8\exp\left({-\delta \eps^2\frac{\dim(\H_1)}{2\norm{B}^2}}\right).
    \end{split}
  \end{align}
  The estimate
  \begin{equation}\label{eq:estC}
    \sigma_1\times\sigma_2(|\langle Cx,y\rangle|>\eps) \leq 8\exp\left({-\delta \eps^2\frac{\dim(\H_2)}{2\norm{C}^2}}\right)
  \end{equation}
  can be shown via the same arguments and we then combine \eqref{eq:estA}, \eqref{eq:estD}, \eqref{eq:estB} and \eqref{eq:estC} to obtain
  \begin{align}\label{eq:ProbNormDist}
    \begin{split}
    \sigma_1\times\sigma_2(\norm{M_{x,y}-\E M} > \eps) &\leq \sigma_1\times\sigma_2(\norm{M_{x,y}-\E M}_F > \eps)\\
    &\leq 32\exp\left({-\delta\eps^2\frac{n_0}{8\max\{4\norm{A}^2, \norm{B}^2, \norm{C}^2, 4\norm{D}^2\}}}\right)\\
    &\leq 32\exp\left({-\delta\eps^2\frac{n_0}{32\norm{\A}^2}}\right),
    \end{split}
  \end{align}
  where $\norm{\cdot}_F$ deontes the Frobenius norm.
  
  Next, we apply \Cref{thm:M1M2} to obtain
  \begin{equation*}
    \dH \big(\sigma(M_{x,y}),\sigma(\E M)\big) \leq \big((\norm{M_{x,y}}+\norm{\E M})\norm{M_{x,y}-\E M}\big)^{\frac{1}{2}},
  \end{equation*}
  where we have $\norm{M_{x,y}}\leq\norm{\A}$ because $M_{x,y} = P\A|_{\mathrm{ran}P}$ , where $P$ is the orthogonal projection to the two-dimensional subspace of $\H_1\times\H_2$ spanned by $\begin{bmatrix} x&0\end{bmatrix}^\intercal$ and $\begin{bmatrix} 0& y\end{bmatrix}^\intercal$. Moreover $\norm{\E M}\leq\norm{\E M}_F \leq \sqrt{2}\norm{\A}$ because the trace of a matrix is the sum of its eigenvalues. Thus,
  \begin{equation*}
    \dH \big(\sigma(M_{x,y}),\sigma(\E M)\big) \leq \big((1+\sqrt{2})\norm{\A}\norm{M_{x,y}-\E M}\big)^{\frac{1}{2}}
  \end{equation*}
  and we conclude by using \eqref{eq:ProbNormDist} that
  \begin{align*}
    \sigma_1\times\sigma_2\big(\dH(\sigma(M_{x,y}),\sigma(\E M)) > \eps\big) &\leq \sigma_1\times\sigma_2\left( \left(\big(1+\sqrt{2}\big)\norm{\A}\norm{M_{x,y}-\E M}\right)^{\frac{1}{2}} >\eps\right)\\
                                                &= \sigma_1\times\sigma_2\left( \norm{M(x,y)-\E M} > \frac{\eps^2}{\big(1+\sqrt{2}\big)\norm{\A}}\right)\\
                                                &\leq 32\exp\left(-\delta\frac{\eps^4}{(3+2\sqrt{2})\norm{\A}^2}\frac{n_0}{32\norm{\A}^2}\right)\\
                                                &= 32\exp\left(-\beta\frac{\eps^4n_0}{\norm{\A}^4}\right)
  \end{align*}
  with $\beta = \frac{\delta}{(96+64\sqrt{2})}$.
\end{proof}

\begin{remark}
  Equation \eqref{eq:estA} can also be interpreted in the context of the numerical range and yields an estimate for the probability of a point in $W(\A)$ that is computed via the random vector sampling method to fall outside of a small neighborhood of the expected value.

  In \cite{MartinssonTropp}, Martinsson and Tropp reformulated a result from \cite{GrattonTitley-Peloquin} such that an exponential bound for the deviation of the estimation of the trace of a matrix $A$ via $\langle Ax,x\rangle$ is obtained. The proof however relies on $A$ to be a self-adjoint positive semi-definite matrix.
\end{remark}

\begin{example}
  Let us consider the matrix
  \begin{equation*}
    \A_{5} = \left[\begin{array}{cccccc|cccccc}
                 2 & 0 & \dots & \dots & \dots & 0 & 1 & 0 & \dots & \dots & \dots & 0 \\
                 0 & \ddots & \ddots &  &  & \vdots & 0 & 0 & \ddots &  &  & \vdots \\
                 \vdots & \ddots & 2 & \ddots &   & \vdots & \vdots  & \ddots & \ddots & \ddots & & \vdots \\
                 \vdots &  & \ddots  & -2 & \ddots  & \vdots & \vdots   &  & \ddots & \ddots & \ddots & \vdots  \\
                 \vdots &  &  & \ddots & \ddots & 0 & \vdots &  &  & \ddots & \ddots & 0 \\
                 0 & \dots & \dots & \dots & 0 & -2 & 0 & \dots & \dots & \dots & 0 & 0 \\
                 \hline
                 0 & 0 & \dots & \dots & \dots & 0 & 1+\iu & 0 & \dots & \dots & \dots & 0 \\
                 0 & \ddots & \ddots &  &  & \vdots  & 0 & \ddots & \ddots &  & & \vdots \\
                 \vdots & \ddots & \ddots & \ddots &  & \vdots & \vdots  & \ddots & 1+\iu & \ddots & & \vdots \\
                 \vdots &  & \ddots  & \ddots & \ddots  & \vdots & \vdots  &  & \ddots & 1-\iu & \ddots & \vdots \\
                     \vdots &  &  & \ddots & 0 & 0 & \vdots &  & & \ddots & \ddots & 0 \\
                 0 & \dots & \dots & \dots & 0 & 1 & 0 & \dots & \dots & \dots & 0 & 1-\iu
               \end{array}\right]
           \end{equation*}
           with $A$, $B$, $C$ and $D$ equally sized. We have $\norm{\A_5} \approx 2.36$ independent of its dimension. In \Cref{fig:Concentration_sampling}, each plot depicts 10.000.000 points in the QNR of a version of $\A_5$ with $\dim(\A_5) = 2^{n}$, $n=2,\dots,7$, that were generated via the random vector sampling method.  The concentration phenomenon proven in \Cref{thm:probability} becomes clearly visible and with an increase of the dimension, the QNR seems to split into two disconnected components, which is not the case as \Cref{fig:Concentration_alg} shows. There, only $168.820$ points have been computed, but they give a much more accurate picture of the QNR.

\begin{figure}[h]
\centering
\includegraphics[width=.72\textwidth]{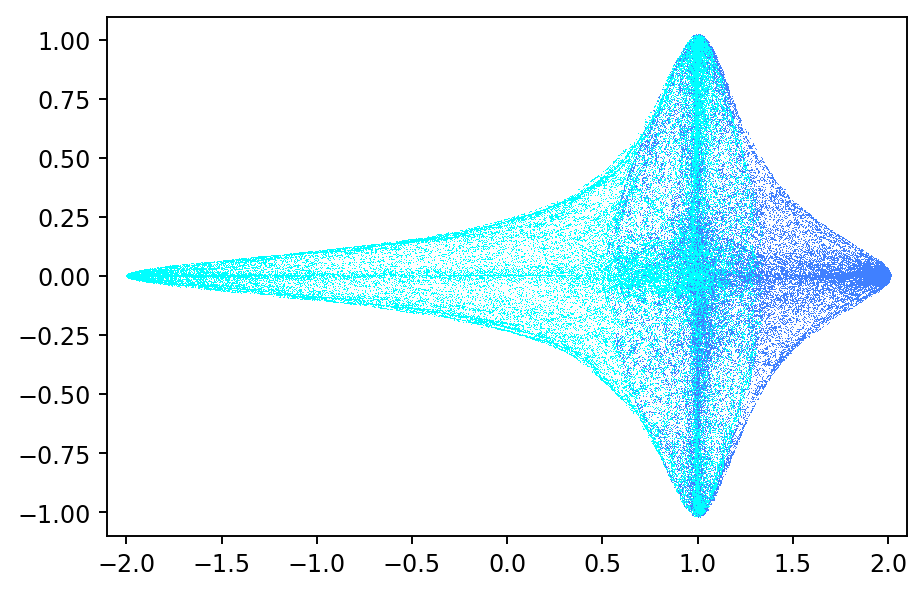}
\captionof{figure}{QNR of $\A_{5}$ with $\dim(\A_5) = 128$ computed with the algorithm in 5 minutes}
\label{fig:Concentration_alg}
\end{figure}
\begin{figure}[H]
\includegraphics[width=\textwidth]{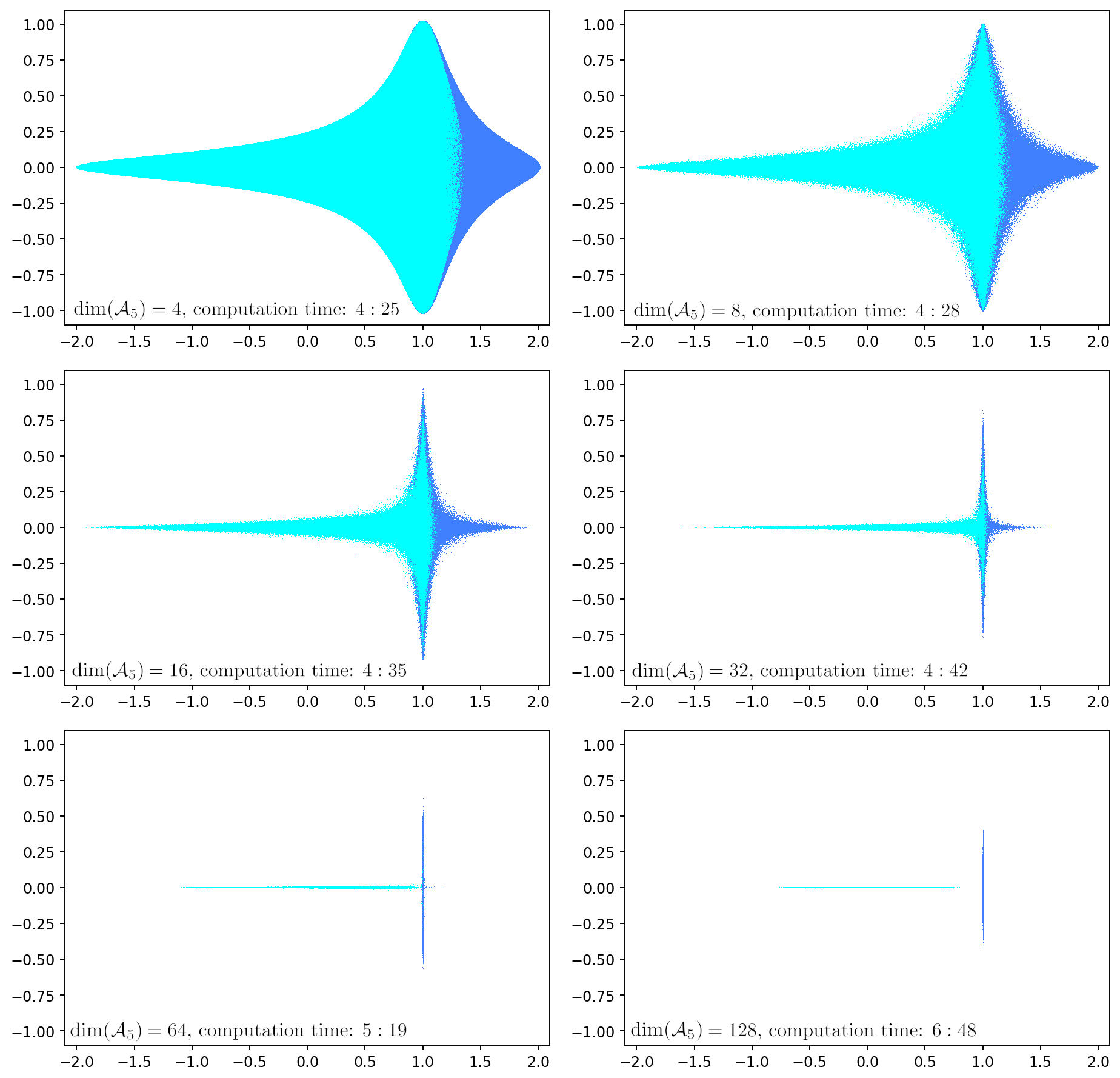}
\captionof{figure}{QNR of $\A_{5}$ computed with the random vector sampling method for different dimensions}
\label{fig:Concentration_sampling}
\end{figure}
\end{example}

\smallskip

\noindent\begin{center}\textbf{Acknowledgement}\end{center}
The authors thank Marco Marletta for suggesting the penalty term in the objective function of the algorithm and for hosting the second author during a research stay at Cardiff University. The authors furthermore thank Jochen Gl\"uck and Julian H\"olz from the University of Wuppertal for valuable input leading to the results of \Cref{sect:concentration} and Andreas Frommer and Karsten Kahl from the University of Wuppertal for fruitful discussions.

\bibliographystyle{abbrv}
\bibliography{mybib}

\begin{thebibliography}{10}

\bibitem{CowenHarel}
C.~Cowen and E.~Harel.
\newblock An effective algorithm for computing the numerical range.
\newblock https://www.math.iupui.edu/~ccowen/Downloads/33NumRange.pdf, 1995.

\bibitem{Fazlollahi}
M.~Fazlollahi.
\newblock Extension of quadratic numerical range of block operator matrices.
\newblock {\em Int. J. Contemp. Math. Sci.}, 3(29-32):1529--1534, 2008.

\bibitem{FrommerJacobKahlWyssZwaan}
A.~Frommer, B.~Jacob, K.~Kahl, C.~Wyss, and I.~Zwaan.
\newblock Krylov type methods for linear systems exploiting properties of the
  quadratic numerical range.
\newblock {\em Electron. Trans. Numer. Anal.}, 53:541--561, 2020.

\bibitem{GrattonTitley-Peloquin}
S.~Gratton and D.~Titley-Peloquin.
\newblock Improved bounds for small-sample estimation.
\newblock {\em SIAM J. Matrix Anal. Appl.}, 39(2):922--931, 2018.

\bibitem{GustafsonRao}
K.~E. Gustafson and D.~K.~M. Rao.
\newblock {\em Numerical range}.
\newblock Universitext. Springer-Verlag, New York, 1997.
\newblock The field of values of linear operators and matrices.

\bibitem{HornJohnson}
R.~A. Horn and C.~R. Johnson.
\newblock {\em Topics in matrix analysis}.
\newblock Cambridge University Press, Cambridge, 1994.
\newblock Corrected reprint of the 1991 original.

\bibitem{JacobTretterTrunkVogt}
B.~Jacob, C.~Tretter, C.~Trunk, and H.~Vogt.
\newblock Systems with strong damping and their spectra.
\newblock {\em Math. Methods Appl. Sci.}, 41(16):6546--6573, 2018.

\bibitem{Johnson}
C.~R. Johnson.
\newblock Numerical determination of the field of values of a general complex
  matrix.
\newblock {\em SIAM J. Numer. Anal.}, 15(3):595--602, 1978.

\bibitem{Kato95}
T.~{Kato}.
\newblock {\em {Perturbation theory for linear operators. Reprint of the corr.
  print. of the 2nd ed. 1980.}}
\newblock Berlin: Springer-Verlag, reprint of the corr. print. of the 2nd ed.
  1980 edition, 1995.

\bibitem{LangerMarkusMatsaevTretter}
H.~Langer, A.~Markus, V.~Matsaev, and C.~Tretter.
\newblock A new concept for block operator matrices: the quadratic numerical
  range.
\newblock {\em Linear Algebra Appl.}, 330(1-3):89--112, 2001.

\bibitem{LangerTretter}
H.~Langer and C.~Tretter.
\newblock Spectral decomposition of some nonselfadjoint block operator
  matrices.
\newblock {\em J. Operator Theory}, 39(2):339--359, 1998.

\bibitem{Linden}
H.~Linden.
\newblock The quadratic numerical range and the location of zeros of
  polynomials.
\newblock {\em SIAM J. Matrix Anal. Appl.}, 25(1):266--284, 2003.

\bibitem{LoiselMaxwell}
S.~Loisel and P.~Maxwell.
\newblock Path-following method to determine the field of values of a matrix
  with high accuracy.
\newblock {\em SIAM J. Matrix Anal. Appl.}, 39(4):1726--1749, 2018.

\bibitem{MartinssonTropp}
P.-G. Martinsson and J.~A. Tropp.
\newblock Randomized numerical linear algebra: foundations and algorithms.
\newblock {\em Acta Numer.}, 29:403--572, 2020.

\bibitem{MilmanSchechtman}
V.~D. Milman and G.~Schechtman.
\newblock {\em Asymptotic theory of finite-dimensional normed spaces}, volume
  1200 of {\em Lecture Notes in Mathematics}.
\newblock Springer-Verlag, Berlin, 1986.

\bibitem{muh-mar12}
A.~{Muhammad} and M.~{Marletta}.
\newblock {Approximation of the quadratic numerical range of block operator
  matrices.}
\newblock {\em {Integral Equations Oper. Theory}}, 74(2):151--162, 2012.

\bibitem{muh-mar13}
A.~{Muhammad} and M.~{Marletta}.
\newblock {A numerical investigation of the quadratic numerical range of
  Hain-L\"ust operators.}
\newblock {\em {Int. J. Comput. Math.}}, 90(11):2431--2451, 2013.

\bibitem{Tretter}
C.~{Tretter}.
\newblock {\em {Spectral theory of block operator matrices and applications.}}
\newblock London: Imperial College Press, 2008.

\end{thebibliography}

\end{document}